\documentclass[12pt]{amsart}
\usepackage{amsmath}
\usepackage{amssymb}
\usepackage{latexsym}

\newcommand{\C}{\mathbb {C}}     
\newcommand{\halmos}{\rule{5pt}{5pt}}

\numberwithin{equation}{section}

\newtheorem{prop}{\bf Proposition}[section]
\newtheorem{thm}[prop]{\bf Theorem}

\theoremstyle{definition}

\begin{document}
\title[Reducibility of symmetry on $q$-Painlev\'e equations]
{Reducibility of symmetry on $q$-Painlev\'e equations}
\author{Chihiro Sato}
\address{Department of Mathematics, Ochanomizu University, 2-1-1 Otsuka, Bunkyo-ku, Tokyo 112-8610, Japan}
\author{Kouichi Takemura}
\address{Department of Mathematics, Ochanomizu University, 2-1-1 Otsuka, Bunkyo-ku, Tokyo 112-8610, Japan}
\email{takemura.kouichi@ocha.ac.jp}
\author{Ohka Yamashita}
\address{Department of Mathematics, Ochanomizu University, 2-1-1 Otsuka, Bunkyo-ku, Tokyo 112-8610, Japan}
\begin{abstract}
It is known that $q$-Painlev\'e equations admit the symmetry of the extended affine Weyl groups.
Kajiwara, Noumi and Yamada developed a unified description of the symmetry by introducing representations of the extended affine Weyl groups.
On the other hand, a parallel translation associated with the extended affine Weyl groups describes the time evolution in Sakai's theory of discrete Painlev\'e equations.
In this paper, we introduce the equivalences to the representations of Kajiwara, Noumi and Yamada.
The equivalences imply reducibility of the representations, and they are related to the expressions of the parallel translations of the $q$-Painlev\'e equations.
\end{abstract}
\maketitle

\section{Introduction}

Painlev\'e equations are nonlinear second-order ordinary differential equations whose solutions have the Painlev\'e property such that movable singularities are only poles.
They sometimes appear in the context of mathematical physics.
Several members of the discrete analogue of the Painlev\'e equation had been discovered individually in the 1990's, and a comprehensive list of the second order discrete Painlev\'e equations was provided by Sakai \cite{Sak}.
Each member of the $q$-difference Painlev\'e equation was labelled by the affine Dynkin diagram from the initial value space and the symmetry, and the symmetry was used to characterize the time evolution of the $q$-Painlev\'e equation, which follows from the geometric realization of Sakai \cite{Sak}.
Several novel results on discrete Painlev\'e equations were obtained by authors in \cite{HDC20,DFS20,DFS22,HDC25}, which are related to special functions and to the initial value spaces.
As an application, they clarified an interesting connection between orthogonal polynomials with a certain deformation of the $q$-Laguerre weight and a member of the $q$-Painlev\'e equations \cite{HDC25}.

A systematic description of the extended affine Weyl group symmetry for each member of discrete Painlev\'e equations was provided in the survey paper \cite{KNY} by Kajiwara, Noumi and Yamada, and it was based on the degenerations from the $E^{(1)}_8$ case.
On the other hand, Yamada and his collaborators developed the description of the discrete Painlev\'e equations in terms of the Lax pairs \cite{Ye,Y,NTY,KNY}.
However, some adjustment would be necessary to fit the time evolution from the Lax pair with the one from the extended affine Weyl group symmetry.
In \cite{TsqP}, the expressions of the time evolution of the $q$-Painlev\'e equations $q$-$P(D^{(1)}_5)$, $q$-$P(E^{(1)}_6)$ and $q$-$P(E^{(1)}_7)$ in terms of the extended affine Weyl group symmetry were reviewed, and the differences between the time evolution by the Lax pair and the one by the extended affine Weyl group symmetry were calculated.
Then, the differences were interpreted by the symmetry of the gauge transformation and the dilation on the linear $q$-difference equation associated to the Lax operator.

In this paper, we investigate the reducibility of the representations of the extended affine Weyl group associated with the $q$-Painlev\'e equations $q$-$P(D^{(1)}_5)$, $q$-$P(E^{(1)}_6)$, $q$-$P(E^{(1)}_7)$ and $q$-$P(E^{(1)}_8)$, which were introduced in \cite{KNY}.
We define an equivalence on the variables on each representation and we show compatibility between the actions on the representation and the equivalence, which indicates reducibility of the representations.
Then, the realization of the time evolution of the $q$-Painlev\'e equation agrees with a composition of the generators of the extended affine Weyl group which corresponds to the parallel translation up to the equivalence for the each case of $q$-$P(D^{(1)}_5)$, $q$-$P(E^{(1)}_6)$, $q$-$P(E^{(1)}_7)$ and $q$-$P(E^{(1)}_8)$.
In other words, the equivalence describes a feature of the reducibility, the expression of the time evolution of the $q$-Painlev\'e equation is affected by the reducibility, and the equivalence recovers the compatibility with the parallel translation of the extended affine Weyl group.

This paper is organized as follows.
In Section \ref{sec:prel}, we review the representations of the extended affine Weyl groups associated with the $q$-Painlev\'e equations $q$-$P(E^{(1)}_6)$, $q$-$P(E^{(1)}_7)$, $q$-$P(D^{(1)}_5)$ and $q$-$P(E^{(1)}_8)$ and we describe realizations of parallel translations for them.
In Section \ref{sec:equiv}, we introduce the equivalences for particular representations of the extended affine Weyl groups, and we obtain compatibility among the actions on the representation, the equivalence and the time evolution.
Note that there is a typo in \cite[Theorem 3.1]{TsqP}, and we correct it in Section \ref{sec:equiv} and in the appendix.

\section{$q$-Painlev\'e equations and extended affine Weyl groups} \label{sec:prel}
We recall the $q$-Painlev\'e equations $q$-$P(E^{(1)}_6)$, $q$-$P(E^{(1)}_7)$, $q$-$P(D^{(1)}_5)$, $q$-$P(E^{(1)}_8)$ and the representations of the extended affine Weyl groups, which was introduced systematically in the paper \cite{KNY} by Kajiwara, Noumi and Yamada.
The representations are defined on the rational functions with the variables $\nu_{1}, \nu_{2}, \dots , \nu_{8} , \kappa _1 , \kappa _2, q, f,g $ with the relation
\begin{equation}
\kappa _1 ^2 \kappa _2^2 = q \nu_{1}\nu_{2}\nu_{3}\nu_{4} \nu_{5}\nu_{6}\nu_{7}\nu_{8} .
\end{equation}  
The principal variables of the $q$-Painlev\'e equations are $f$ and $g$, and the variables $\nu_{1}, \nu_{2}, \dots , \nu_{8} , \kappa _1 , \kappa _2$ are supplementary.
Let $\overline{a} $ (resp. $\underline{a}$) be the consequence of the time evolution $t \mapsto qt$ (resp. the inverse time evolution $t \mapsto t/q$) to the variable $a$.
The time evolution for the parameters $\kappa _1 , \kappa _2 , \nu_{1}, \nu_{2}, \dots , \nu_{8} ,  q$ is given by 
\begin{align}
& ( \overline{\kappa _1} , \overline{\kappa _2} , \overline{\nu_{1}} , \overline{\nu_{2}} , \dots , \overline{\nu_{8}} ,  \overline{q} ) =  ( \kappa _1 /q , q \kappa _2 , \nu_{1}, \nu_{2}, \dots , \nu_{8} ,  q )  . 
\label{eq:Tolinenu}
\end{align}
The time evolution for the parameters $f$, $g$ are determined individually for each $q$-Painlev\'e equation, and we describe it for several $q$-Painlev\'e equations in Sections \ref{sec:WE6}--\ref{sec:WE8}.

\subsection{} \label{sec:WE6}
The $q$-Painlev\'e equation of type $E^{(1)}_{6} $ ($q$-$P(E^{(1)}_{6})$) given in \cite{KNY} is written as
\begin{align}
 & \frac{(fg-1)(\overline{f}g-1)}{f\overline{f}}=\frac{(g- 1 / \nu_{1} ) (g- 1 / \nu_{2} ) (g- 1 / \nu_{3} ) (g- 1 / \nu_{4} )}{ (g- \nu_{5}/\kappa_{2} ) (g- \nu_{6}/\kappa_{2} )}, \label{eq:qPE16} \\
 & \frac{(fg-1)(f\underline{g}-1)}{g\underline{g}}=\frac{(f-\nu_{1}) (f-\nu_{2}) (f-\nu_{3}) (f-\nu_{4})}{ (f- \kappa_{1}/\nu_{7} )(f- \kappa_{1}/\nu_{8} )} . \nonumber
\end{align}
Let $\widetilde{W}(E^{(1)}_{6})$ be the extended affine Weyl group, which is generated by $s_0$, $s_1$, $s_2$, $s_3$, $s_4$, $s_5$, $s_6$, $\pi _1$, $\pi _2$ and the defining relations are 
\begin{align}
& \pi _1^2 = \pi_2^2 = (\pi _1 \pi _2)^3 = 1, \quad  s_{i}^{2} = 1 , \quad i = 0, \dots , 6 , \\
& s_{i}s_{j}s_i = s_j s_i s_j , \quad  \{ i,j \} = \{ 1,2\} , \{ 2,3 \}, \{ 3,4 \} , \{ 4,5 \} , \{ 3,6 \}, \{ 0,6 \}, \nonumber \\
& s_{i}s_{j} = s_{j}s_{i} , \quad  \mbox{ otherwise. }  \nonumber \\
& \pi _1 s_1 = s_5 \pi _1, \quad \pi _1 s_2 = s_4 \pi _1, \quad \pi _1 s_j = s_j \pi _1, \quad j=0,3,6,\nonumber \\
&  \pi _2 s_0 = s_1 \pi _2, \quad \pi _2 s_2 = s_6 \pi _2 , \quad \pi _2 s_j = s_j \pi _2, \quad j=3,4,5 . \nonumber 
\end{align}
Note that the Dynkin diagram of type $E^{(1)}_{6} $ is described as follows.
\\
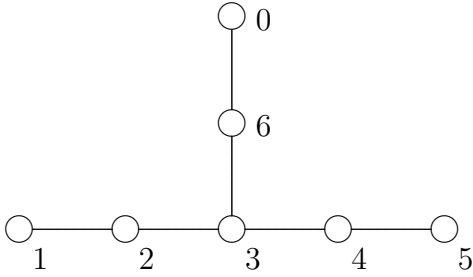
\begin{picture}(0,120)(0,-10)
\put(20,20){\circle{10}}
\qbezier(25,20)(40,20)(55,20)
\put(60,20){\circle{10}}
\qbezier(65,20)(80,20)(95,20)
\put(100,20){\circle{10}}
\qbezier(105,20)(120,20)(135,20)
\put(140,20){\circle{10}}
\qbezier(145,20)(160,20)(175,20)
\put(180,20){\circle{10}}
\qbezier(100,25)(100,40)(100,55)
\put(100,60){\circle{10}}
\qbezier(100,65)(100,80)(100,95)
\put(100,100){\circle{10}}
\put(25,5){$1$}
\put(65,5){$2$}
\put(105,5){$3$}
\put(145,5){$4$}
\put(185,5){$5$}
\put(109,55){$6$}
\put(109,95){$0$}
\end{picture}

The Weyl group symmetry of the $q$-Painlev\'e equation $q$-$P(E^{(1)}_6 )$ is given by the following action of the operators $s_0, \dots ,s_6$, $\pi _1$ and $\pi _2$ for the parameters $(\kappa_{1}, \kappa_{2}, \nu_{1}, \dots , \nu_{8}) $ and $(f, g) $ in \cite{KNY}.
\begin{align}
s_0 : & \: \nu_7 \leftrightarrow \nu_8 , \quad s_1 : \: \nu_5 \leftrightarrow \nu_6  , \quad s_3 : \: \nu_1 \leftrightarrow \nu_2 , \quad s_4 : \: \nu_2 \leftrightarrow \nu_3 , \quad s_5 : \: \nu_3 \leftrightarrow \nu_4 , \label{eq:AffineWeylActionE6} \\
 s_2 : & \: \nu_1 \rightarrow\frac{\kappa_2}{\nu_6}, \;
 \nu_6 \rightarrow \frac{\kappa_2}{\nu_1}, \;
 \kappa_1 \rightarrow \frac{\kappa_1\kappa_2}{\nu_1\nu_6}, \;
 f \rightarrow f\, \frac{\kappa_2 ( \nu_1 g-1)  }{ (\nu _1 \nu_6 - \kappa_2  )fg + \nu _1 \kappa_2 g - \nu _1 \nu_6 } ,\nonumber \\
 s_6 : & \: \nu_1 \rightarrow\frac{\kappa_1}{\nu_7}, \; 
 \nu_7 \rightarrow \frac{\kappa_1}{\nu_1}, \;
 \kappa_2 \rightarrow \frac{\kappa_1\kappa_2}{\nu_1\nu_7}, \;
  g \rightarrow g\, \frac{\nu _7 ( \nu_1 - f )  }{ \kappa_1  -  \nu _7 f + ( \nu _1 \nu_7 - \kappa_1 )fg } , \nonumber \\
  \pi _1 : & \: q \rightarrow 1/q, \: \nu_1 \rightarrow \nu_2 / \kappa_2 , \:  \nu_2 \rightarrow \nu_1 / \kappa_2 ,  \: \nu_3 \rightarrow 1/\nu_6,  \: \nu_4 \rightarrow 1/\nu_5, \: \nu_5 \rightarrow 1/\nu_4,  \nonumber \\
& \: \nu_6 \rightarrow 1/\nu_3, \: \nu_7 \rightarrow 1/\nu_7, \: \nu_8 \rightarrow 1/\nu_8, \: \kappa_1 \rightarrow \nu_1 \nu_2 /(\kappa_1 \kappa_2 ) , \: \kappa_2 \rightarrow 1/\kappa_2 ,\nonumber \\
& \: f \rightarrow \frac{\nu _1 \nu_2 ( 1- fg )}{ \kappa_2 \{ \nu _1 \nu_2 g + f- (\nu _1 + \nu _2 ) fg \} } , \: g \rightarrow \kappa_2 g , \nonumber \\
  \pi _2 : & \: q \rightarrow 1/q, \: \nu_1 \rightarrow 1/\nu_1, \:  \nu_2 \rightarrow 1/\nu_2,  \: \nu_3 \rightarrow 1/\nu_3,  \: \nu_4 \rightarrow 1/\nu_4, \: \nu_5 \rightarrow 1/\nu_8, \nonumber \\
& \: \nu_6 \rightarrow 1/\nu_7, \: \nu_7 \rightarrow 1/\nu_6, \: \nu_8 \rightarrow 1/\nu_5, \: \kappa_1 \rightarrow 1/\kappa_2 , \: \kappa_2 \rightarrow 1/\kappa_1 , \: f \leftrightarrow  g . \nonumber 
\end{align}
The omitted variables are invariant by the action, e.g.~$s_2 (g) =g$.
Then, we can confirm that these operations satisfy the relations of the extended affine Weyl group $\widetilde{W}(E^{(1)}_{6})$.

A candidate for the parallel translation corresponding the time evolution of $q$-Painlev\'e equation $q$-$P(E^{(1)}_6)$ is written as $(\pi _1 \pi _2 s_4 s_5 s_3 s_6 s_4 s_3 s_0 s_6 )^2 $, which was essentially used in Tsuda \cite{Td} (see also \cite{TsqP}).
Set $s= \pi _1 \pi _2 s_4 s_5 s_3 s_6 s_4 s_3 s_0 s_6 $.
As discussed in the proof of \cite[Theorem 3.1]{TsqP}, we have
\begin{align}
& s (\nu_1) = \kappa_2 /\nu_4 , \: s (\nu_2) =\kappa_2 /\nu_3 , \: s (\nu_3) = \kappa_2/\nu_2  , \: s (\nu_4) = \kappa_2/\nu_1,  \: s (\nu_5) = \nu_8 , \\
& s (\nu_6) =\nu_7 , \: s (\nu_7 ) = \kappa_2 /\nu_5, \: s (\nu_8 ) = \kappa_2 /\nu_6 , \: s (\kappa _1 ) = \kappa_2 , \: s  (\kappa _2 ) =  q  \nu_7 \nu_8 \kappa_2 / \kappa_1 ,  \nonumber \\
& s (f)  = \kappa_2 g , \; \frac{1}{\kappa_2 s (g)}  = g +\frac{f(g -1/\nu_1)(g -1/\nu_2)(g -1/\nu_3)(g -1/\nu_4)}{(1-fg)(g-\nu_5/\kappa_2)(g-\nu_6/\kappa_2)} , \nonumber
\end{align}
and 
\begin{align}
& s ^2 (\nu_1) = q \nu_1 \nu_7 \nu_8 /\kappa_1 , \: s^2 (\nu_2) = q \nu_2 \nu_7 \nu_8 /\kappa_1 , \: s^2 (\nu_3) = q \nu_3 \nu_7 \nu_8 / \kappa_1  , \: s ^2  (\nu_4) =  q \nu_4 \nu_7 \nu_8 / \kappa_1 , \label{eq:s^2E6} \\
& s^2 (\nu_5) = \kappa_2 /\nu_6 , \: s^2 (\nu_6) =\kappa_2 / \nu_5  , \: s^2 (\nu_7 ) = q \nu_7 \kappa_2 / \kappa_1 , \: s ^2(\nu_8 ) =q \nu_8 \kappa_2 / \kappa_1 , \nonumber \\
& s ^2 (\kappa _1 ) =  q \nu_7 \nu_8 \kappa_2/ \kappa_1  , \: s ^2 (\kappa _2 ) = q ^2 \nu_7 \nu_8 \kappa_2^2 /(\nu_5 \nu_6 \kappa_1 ) , \nonumber \\
& \frac{1}{s^2 (f)}   =  \frac{ \kappa_1}{ q  \nu_7 \nu_8 } \Bigl\{ g +\frac{f(g -1/\nu_1)(g -1/\nu_2)(g -1/\nu_3)(g -1/\nu_4)}{(1-fg)(g-\nu_5/\kappa_2)(g-\nu_6/\kappa_2)} \Bigr\}  , \nonumber \\
& \frac{1}{s^2 (g)}  = s^2 (f ) + \frac{\kappa_2  g ( s^2 (f )  - s^2 (\nu_4))( s^2 (f )  -s^2 (\nu_3 ))( s^2 (f )  - s^2 (\nu_2))( s^2 (f )  -s^2(\nu_1))}{ (s(\kappa _2 )- \kappa_2  g  s^2 (f ))( s^2 (f ) - \nu_8 )( s^2 (f ) - \nu_7 )} . \nonumber 
\end{align}
Hence, the action of the operator $(\pi _1 \pi _2 s_4 s_5 s_3 s_6 s_4 s_3 s_0 s_6 )^2 $ does not coincide with the time evolution in equation (\ref{eq:qPE16}), but the two expressions resemble each other.
We investigate the difference of them.
Let $T$ be the time evolution of $q$-$P(E^{(1)}_6)$ given in equation (\ref{eq:qPE16}).
It follows from equations (\ref{eq:Tolinenu}) and (\ref{eq:qPE16}) that
\begin{align}
& \displaystyle T(\nu _i)= \nu _i \; (i=1,2,\dots ,8) , \;  T(\kappa _1 ) =  \kappa _1  / q , \; T(\kappa _2 ) = q \kappa _2 , \label{eq:E6T}
\end{align}
and
\begin{align}
& \frac{(T(f) g-1)(fg-1)}{T(f) f} = \frac{(g -1/\nu_1)(g -1/\nu_2)(g -1/\nu_3)(g -1/\nu_4)}{(g-\nu_5/\kappa_2)(g-\nu_6/\kappa_2)}, \label{eq:E6TfTg} \\
& \frac{(T(f) g- 1 )(T(f) T(g) - 1 )}{g T (g)} = \frac{(T(f) - \nu _1 ) (T(f) - \nu _2 ) (T(f) - \nu _3 ) (T(f) - \nu _4 )}{(T(f)-T(\kappa_1 / \nu_7 ))(T(f)-T(\kappa_1 / \nu_8 ))} . \nonumber
\end{align}
To clarify the difference between the time evolution $T$ and the parallel translation $s^2$, we define the operator $\Xi $ by $T= \Xi s^2$ $(\Leftrightarrow \Xi= T s^{-2})$.
It follows from a straightforward calculation that
\begin{align}
& \Xi (\nu _1) = \frac{ \nu_ 5 \nu_ 6 }{\kappa_2 } \nu _1 , \: \Xi (\nu _2) = \frac{ \nu_ 5 \nu_ 6 }{\kappa_2 }  \nu _2 , \: \Xi (\nu _3) = \frac{ \nu_ 5 \nu_ 6 }{\kappa_2 }  \nu _3 , \: \Xi (\nu _4) = \frac{ \nu_ 5 \nu_ 6 }{\kappa_2 }  \nu _4 , \: \Xi (\nu _5) = \frac{\kappa_1 }{q \kappa_2 } \nu _5, \label{eq:E6Xinu} \\
& \Xi (\nu _6) = \frac{\kappa_1 }{q \kappa_2 } \nu _6 , \: \Xi (\nu _7) = \frac{\kappa_1 }{q \nu_8 } , \: \Xi (\nu _8) = \frac{\kappa_1 }{q \nu_7 } , \: \Xi (\kappa _1) =  \frac{\kappa ^2 _1 \nu_5 \nu_6 }{q \kappa_2 \nu_7 \nu_8 } , \: \Xi (\kappa _2) = \frac{\kappa_1 \nu_5 \nu_6 }{q \kappa_2} .\nonumber
\end{align}
By combining equation (\ref{eq:s^2E6}) with equation (\ref{eq:E6TfTg}), we have
\begin{align}
& \Xi (f) = \frac{ \nu_ 5 \nu_ 6 }{\kappa_2 } f, \quad \Xi (g) = \frac{\kappa_2 }{ \nu_ 5 \nu_ 6 } g . \label{eq:E6Xifg} 
\end{align}
The expressions of equations (\ref{eq:E6Xinu}) and (\ref{eq:E6Xifg}) motivated us to introduce the equivalence in Section \ref{sec:E16}, which is related to the representation of $\widetilde{W}(E^{(1)}_{6})$.

\subsection{}
The $q$-Painlev\'e equation of type $E^{(1)}_{7} $ ($q$-$P(E^{(1)}_{7})$) given in \cite{KNY} is written as
\begin{align}
& \frac{(fg- \kappa_{1}/\kappa_{2} ) (\overline{f}g- \kappa_{1}/(q\kappa_{2}) )}{(fg-1)(\overline{f}g-1)}=\frac{(g- \nu_{5}/\kappa_{2} ) (g- \nu_{6}/\kappa_{2} ) (g- \nu_{7}/\kappa_{2} ) (g- \nu_{8}/\kappa_{2} )}{ (g - 1/\nu_{1} )(g - 1/\nu_{2} )(g - 1/\nu_{3} )(g - 1/\nu_{4} )}, \label{eq:qPE17} \\
& \frac{(fg-\kappa_{1}/\kappa_{2} ) (f\underline{g}- q\kappa_{1}/\kappa_{2} )}{(fg-1)(f\underline{g}-1)}=\frac{(f- \kappa_{1}/\nu_{5})(f- \kappa_{1}/\nu_{6})(f- \kappa_{1}/\nu_{7})(f- \kappa_{1}/\nu_{8})}{(f-\nu_{1})(f-\nu_{2})(f-\nu_{3})(f-\nu_{4})}  . \nonumber 
\end{align}
The Weyl group symmetry of the $q$-Painlev\'e equation $q$-$P(E^{(1)}_7 )$ is given by the following action of the operators $s_0, \dots ,s_7$ and $\pi $ for the parameters $(\kappa_{1}, \kappa_{2}, \nu_{1}, \dots , \nu_{8}) $ and $(f, g) $ in \cite{KNY}.
\begin{align}
 s_0 : & \: \kappa_1 \leftrightarrow \kappa_2 ,  \: f \rightarrow 1/g , \: g \rightarrow 1/f ,\label{eq:AffineWeylActionE7} \\
  s_1 : & \: \nu_3 \leftrightarrow \nu_4 , \quad   s_2 : \: \nu_2 \leftrightarrow \nu_3 , \quad  s_3 : \nu_1 \leftrightarrow \nu_2 ,\nonumber \\
  s_5 : & \: \nu_5 \leftrightarrow \nu_6 , \quad   s_6 : \: \nu_6 \leftrightarrow \nu_7 , \quad   s_7 : \: \nu_7 \leftrightarrow \nu_8 \nonumber \\
  s_4 : & \: \nu_1 \rightarrow\frac{\kappa_2}{\nu_5}, \;
  \nu_5 \rightarrow \frac{\kappa_2}{\nu_1}, \;
  \kappa_1 \rightarrow \frac{\kappa_1\kappa_2}{\nu_1\nu_5}, \nonumber \\
& \:    f \rightarrow \frac{- \kappa_2 ( \nu_1 \nu_5 - \kappa_1 )f g - \nu_5 ( \kappa_1 - \kappa_2 )f + \kappa_1 ( \nu_1 \nu_5 - \kappa_2 )}{\nu _5 \{ - ( \nu_1 \nu_5 - \kappa_2 )f g + \nu_1 ( \kappa_1 - \kappa_2 )g + ( \nu_1 \nu_5 - \kappa_1 ) \}} , \nonumber \\
  \pi  : & \: q \rightarrow 1/q, \: \nu_1 \rightarrow 1/\nu_5, \:  \nu_2 \rightarrow 1/\nu_6,  \: \nu_3 \rightarrow 1/\nu_7,  \: \nu_4 \rightarrow 1/\nu_8, \: \nu_5 \rightarrow 1/\nu_1, \: \nu_6 \rightarrow 1/\nu_2, \nonumber \\
& \: \nu_7 \rightarrow 1/\nu_3, \: \nu_8 \rightarrow 1/\nu_4, \: \kappa_1 \rightarrow 1/\kappa_1 , \: \kappa_2 \rightarrow 1/\kappa_2 , \: f \rightarrow f/\kappa_1  , \: g \rightarrow \kappa_2 g , \nonumber 
\end{align}
Then, we can confirm that these operations satisfy the relations of the extended affine Weyl group $\widetilde{W}(E^{(1)}_{7})$ whose Dynkin diagram is as follows.\\
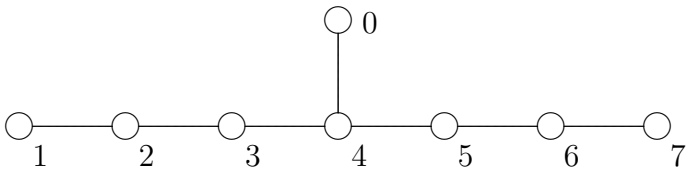
\begin{picture}(0,70)(0,0)
\put(20,20){\circle{10}}
\qbezier(25,20)(40,20)(55,20)
\put(60,20){\circle{10}}
\qbezier(65,20)(80,20)(95,20)
\put(100,20){\circle{10}}
\qbezier(105,20)(120,20)(135,20)
\put(140,20){\circle{10}}
\qbezier(145,20)(160,20)(175,20)
\put(180,20){\circle{10}}
\qbezier(185,20)(200,20)(215,20)
\put(220,20){\circle{10}}
\qbezier(225,20)(240,20)(255,20)
\put(260,20){\circle{10}}
\qbezier(140,25)(140,40)(140,55)
\put(140,60){\circle{10}}
\put(25,5){$1$}
\put(65,5){$2$}
\put(105,5){$3$}
\put(145,5){$4$}
\put(185,5){$5$}
\put(225,5){$6$}
\put(265,5){$7$}
\put(149,55){$0$}
\end{picture}

A candidate for the parallel translation corresponding the time evolution of $q$-Painlev\'e equation $q$-$P(E^{(1)}_7)$ is written as $ ( s_4 s_5 s_1 s_4 s_6 s_5 s_1 s_2 s_4 s_7 s_6 s_5 s_1 s_2 s_3 s_4 s_0  )^2 $, which was essentially used in Tsuda \cite{Td} (see also \cite{TsqP}).
Set $s= s_4 s_5 s_1 s_4 s_6 s_5 s_1 s_2 s_4 s_7 s_6 s_5 s_1 s_2 s_3 s_4 s_0 $.
As discussed in the proof of \cite[Theorem 4.1]{TsqP}, we have
\begin{align}
& s(\nu_i) = \kappa_2 /\nu_{9-i} \:  (i=1,\dots ,8), \: s (\kappa _1 ) = \kappa_2 , \: s (\kappa _2 ) = q  \kappa_2^2 / \kappa_1 , \: s (f) = 1/g , \label{eq:E7sg} \\
& \frac{( s(g)/g -\kappa_1/(q \kappa_2))(fg -1)}{(s(g)/g -1)(f g -\kappa_1 /\kappa_2 )} = \frac{ \kappa_1/(q \kappa_2 ) (g -1/\nu_1)(g -1/\nu_2)(g -1/\nu_3)( g -1/\nu_4 )}{ (g - \nu_5 /\kappa_2)(g - \nu_6 /\kappa_2)(g - \nu_7 /\kappa_2 )(g - \nu_8 /\kappa_2 )} , \nonumber
\end{align}
and 
\begin{align}
& s^2(\nu _j) = q \nu _j \kappa_2 /\kappa_1 \: (j=1,\dots 8), \: s^2 (\kappa_1) = q \kappa_2^2 /\kappa_1, \: s^2 (\kappa_2 ) = q^3 \kappa_2^3 /\kappa_1^2, \\
& \frac{(  q \kappa_2/\kappa_1  - g s^2(f) )(fg -1)}{(1-g s^2 (f))(f g -\kappa_1 /\kappa_2 )} = \frac{ (g -1/\nu_1)(g -1/\nu_2)(g -1/\nu_3)( g -1/\nu_4 )}{ (g - \nu_5 /\kappa_2)(g - \nu_6 /\kappa_2)(g - \nu_7 /\kappa_2 )(g - \nu_8 /\kappa_2 )} , \nonumber \\
& \frac{( s^2(g)s^2 (f) -\kappa_1/\kappa_2/q^2 )( g s^2 (f) -1)}{(s^2(g)s^2 (f) -1)(g s^2 (f) -q \kappa_2/\kappa_1 )} \nonumber \\
& = \frac{ (s^2 (f) - \kappa_2 /\nu_5 ) (s^2 (f) - \kappa_2 /\nu_6 ) (s^2 (f) - \kappa_2 /\nu_7 ) (s^2 (f) - \kappa_2 /\nu_8 )  }{ ( s^2 (f) -s^2 ( \nu_1 )) ( s^2 (f) -s^2 ( \nu_2 )) ( s^2 (f) -s^2 ( \nu_3 )) ( s^2 (f) -s^2 ( \nu_4 )) } . \nonumber 
\end{align}
Hence, the action of the operator $( s_4 s_5 s_1 s_4 s_6 s_5 s_1 s_2 s_4 s_7 s_6 s_5 s_1 s_2 s_3 s_4 s_0 )^2 $ does not coincide with the time evolution in equation (\ref{eq:qPE17}), but the two expressions resemble each other.
It was established in \cite{TsqP} that the time evolution of the $q$-Painlev\'e equation of type $E^{(1)}_7 $ is expressed as the following theorem.
\begin{thm} $($\cite[Theorem 4.1 (i)]{TsqP}$)$ \label{thm:TsqPE7}
Let $\Xi $ be the transformation of the parameters defined by
\begin{align}
& (\nu_1 , \nu_2 , \nu_3 , \nu_4 ,\nu_5 ,\nu_6 , \nu_7 ,\nu_8 ;\kappa_1 ,\kappa_2 ; f ,g) \label{eq:E7Xi} \\
& \mapsto  \Bigl( \frac{\kappa_1  }{ q \kappa_2 } \nu _1 ,\frac{\kappa_1  }{ q \kappa_2 } \nu _2 ,\frac{\kappa_1  }{ q \kappa_2 } \nu _3 ,\frac{\kappa_1  }{ q \kappa_2 } \nu _4 ,\frac{\kappa_1  }{ q \kappa_2 } \nu _5 ,\frac{\kappa_1  }{ q \kappa_2 } \nu _6 ,\frac{\kappa_1  }{ q \kappa_2 } \nu _7 ,\frac{\kappa_1  }{ q \kappa_2 } \nu _8  ; \nonumber \\
& \qquad \frac{\kappa_1^2}{q^2 \kappa _2^2 } \kappa_1,\frac{\kappa_1^2 }{q^2 \kappa ^2_2 } \kappa _2; \frac{\kappa_1}{q \kappa _2} f , \frac{q \kappa_2 }{\kappa_1 } g \Bigr) ,\nonumber
\end{align}
Let $T$ be the transformation of the parameters defined by 
\begin{equation}
T= \Xi \cdot ( s_4 s_5 s_1 s_4 s_6 s_5 s_1 s_2 s_4 s_7 s_6 s_5 s_1 s_2 s_3 s_4 s_0  )^2 . \label{eq:E7TXis}
\end{equation}
Then 
\begin{align}
& \frac{(T(f) g -\kappa_1/(q \kappa_2 ) )(f g -\kappa_1 /\kappa_2 )}{( T(f) g-1 )(fg -1)} = \frac{ ( g -  \nu _5 / \kappa_2 ) ( g -  \nu _6 / \kappa_2 ) ( g -  \nu _7 / \kappa_2 ) ( g -  \nu _8 / \kappa_2 ) }{(g -1/ \nu _1 ) (g -1/ \nu _2 ) (g -1/ \nu _3 ) (g -1/ \nu _4 ) } , \label{eq:E7TfTg} \\
& \frac{( T(g) T(f) -\kappa_1/(q^2 \kappa_2) )( T(f) g - \kappa _1 /( q \kappa_2 ))}{(T(g)T(f) -1)(T (f) g - 1 )} \nonumber \\
& = \frac{ (T (f) - T(\kappa_1 / \nu_5 )) (T (f) - T(\kappa_1 /\nu_6 )) (T (f) - T( \kappa_1 /\nu_7 )) (T (f) - T(\kappa_1 / \nu_8 )) }{ ( T (f) - \nu_1 ) ( T (f) - \nu_2 ) ( T (f) - \nu_3 ) ( T (f) - \nu_4 )}. \nonumber
\end{align}
Namely the operator $T$ represents the time evolution of $q$-$P(E^{(1)}_7)$ (see equation (\ref{eq:qPE17})).
On the other parameters, we have
\begin{align}
& \displaystyle T(\nu _i)= \nu _i \; (i=1,2,\dots ,8) , \;  T(\kappa _1 ) =  \kappa _1  / q , \; T(\kappa _2 ) = q \kappa _2  . \label{eq:E7T} 
\end{align}
\end{thm}

\subsection{}
The $q$-Painlev\'e equation of type $D^{(1)}_5 $ ($q$-$P(D^{(1)}_5)$) given in \cite{KNY} is written as
\begin{equation}
f\overline{f} = \nu_3\nu_4\, \frac{ ( g - \nu_5 /\kappa_2 ) ( g - \nu_6/\kappa_2 )}{ ( g - 1/\nu_1 ) ( g - 1/\nu_2) }, \quad
g\underline{g} = \frac{1}{\nu_1\nu_2}\frac{( f - \kappa_1 /\nu_7 ) ( f - \kappa_1/\nu_8 )}{(f - \nu_3)(f - \nu_4)}. \label{eq:qPD15}
\end{equation}
The Weyl group symmetry of the $q$-Painlev\'e equation $q$-$P(D^{(1)}_5 )$ is given by the following action of the operators $s_0, \dots ,s_5$, $\pi _1$ and $\pi _2$ for the parameters $(\kappa_{1}, \kappa_{2}, \nu_{1}, \dots , \nu_{8}) $ and $(f, g) $ in \cite{KNY}.
\begin{align}
s_0 : & \: \nu_7 \leftrightarrow \nu_8 , \quad s_1 : \: \nu_3 \leftrightarrow \nu_4 , \quad  s_4 : \: \nu_1 \leftrightarrow \nu_2 , \quad  s_5 : \: \nu_5 \leftrightarrow \nu_6 , \label{eq:AffineWeylActionD5} \\
s_2 : & \: \nu_3 \rightarrow\frac{\kappa_1}{\nu_7}, \; 
 \nu_7 \rightarrow \frac{\kappa_1}{\nu_3}, \;
 \kappa_2 \rightarrow \frac{\kappa_1\kappa_2}{\nu_3\nu_7}, \;
 g \rightarrow g\, \frac{f - \nu_3}{f - \kappa_1 /\nu_7} , \nonumber \\
s_3 : & \: \nu_1 \rightarrow\frac{\kappa_2}{\nu_5}, \;
 \nu_5 \rightarrow \frac{\kappa_2}{\nu_1}, \;
 \kappa_1 \rightarrow \frac{\kappa_1\kappa_2}{\nu_1\nu_5}, \;
 f \rightarrow f\, \frac{g - 1/\nu_1}{g - \nu_5/\kappa_2} , \nonumber \\
  \pi _1 : & \: q \rightarrow 1/q, \: \nu_1 \rightarrow 1/\nu_1, \:  \nu_2 \rightarrow 1/\nu_2,  \: \nu_3 \rightarrow 1/\nu_7,  \: \nu_4 \rightarrow 1/\nu_8, \: \nu_5 \rightarrow 1/\nu_5, \: \nu_6 \rightarrow 1/\nu_6, \nonumber \\
& \: \nu_7 \rightarrow 1/\nu_3, \: \nu_8 \rightarrow 1/\nu_4, \: \kappa_1 \rightarrow 1/\kappa_1 , \: \kappa_2 \rightarrow 1/\kappa_2 , \: f \rightarrow f/\kappa_1 , \: g \rightarrow 1/g , \nonumber \\
  \pi _2 : & \: q \rightarrow 1/q, \: \nu_1 \rightarrow 1/\nu_7, \:  \nu_2 \rightarrow 1/\nu_8,  \: \nu_3 \rightarrow 1/\nu_5,  \: \nu_4 \rightarrow 1/\nu_6, \: \nu_5 \rightarrow 1/\nu_3, \: \nu_6 \rightarrow 1/\nu_4, \nonumber \\
& \: \nu_7 \rightarrow 1/\nu_1, \: \nu_8 \rightarrow 1/\nu_2, \: \kappa_1 \rightarrow 1/\kappa_2 , \: \kappa_2 \rightarrow 1/\kappa_1 , \: f \rightarrow 1/(\kappa_2 g) , \: g \rightarrow \kappa_1/f . \nonumber 
\end{align}
Then, we can confirm that these operations satisfy the relations of the extended Weyl group $\widetilde{W}(D^{(1)}_5)$ whose Dynkin diagram is as follows.\\
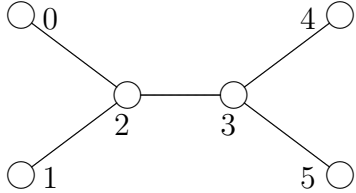
\begin{picture}(0,100)(0,0)
\put(70,80){\circle{10}}
\put(70,20){\circle{10}}
\put(190,80){\circle{10}}
\put(190,20){\circle{10}}
\put(110,50){\circle{10}}
\put(150,50){\circle{10}}
\qbezier(74,77)(90,65)(106,53)
\qbezier(74,23)(90,35)(106,47)
\qbezier(186,77)(170,65)(154,53)
\qbezier(186,23)(170,35)(154,47)
\qbezier(115,50)(130,50)(145,50)
\put(78,75){$0$}
\put(78,15){$1$}
\put(175,75){$4$}
\put(175,15){$5$}
\put(105,35){$2$}
\put(145,35){$3$}
\end{picture}

A candidate for the parallel translation corresponding the time evolution of $q$-Painlev\'e equation $q$-$P(D^{(1)}_5)$ is written as $(\pi _2 \pi _1 s_2 s_1 s_0 s_2 )^2 $, which was essentially written in Sakai \cite{Sak}.
Set $s=  \pi _2 \pi _1 s_2 s_1 s_0 s_2 $.
As discussed in the proof of \cite[Theorem 2.1]{TsqP}, we have
\begin{align}
& s (\nu_1) = \nu_7 , \: s (\nu_2) =\nu_8, \: s (\nu_3) = \kappa_2 /\nu_6 , \: s (\nu_4) = \kappa_2 / \nu_5 , \\
& s (\nu_5) = \nu_3 , \: s (\nu_6) =\nu_4 , \: s (\nu_7 ) = \kappa_2 /\nu_2, \: s (\nu_8 ) = \kappa_2 /\nu_1 , \nonumber \\
& s (\kappa _1 ) = \kappa_2 , \: s  (\kappa _2 ) = \kappa_1 \kappa_2^2 /(\nu_1 \nu_2 \nu_5 \nu_6) = q \nu_3 \nu_4 \nu_7 \nu_8 / \kappa_1 , \: s (f) =1/g,  \nonumber \\
& s (g ) = \frac{\kappa_1  f }{q \nu_3 \nu_4 \nu_7 \nu_8 } \frac{( g -1/\nu_1 )( g -1 /\nu_2 )}{(g -\nu _5 /\kappa _2 )(g  - \nu _6 /\kappa _2)}, \nonumber 
\end{align}
and 
\begin{align}
& s ^2 (\nu_1) = \kappa_2 /\nu_2 , \: s^2 (\nu_2) =\kappa_2 /\nu_1 , \: s^2 (\nu_3) = q \nu_3 \nu_7 \nu_8 / \kappa_1  , \: s ^2  (\nu_4) =  q \nu_4 \nu_7 \nu_8 / \kappa_1 , \\
& s^2 (\nu_5) = \kappa_2 /\nu_6 , \: s^2 (\nu_6) =\kappa_2 / \nu_5  , \: s^2 (\nu_7 ) = q \nu_3 \nu_4 \nu_7 / \kappa_1 , \: s ^2(\nu_8 ) = q \nu_3 \nu_4 \nu_8 / \kappa_1 , \nonumber \\
& s ^2 (\kappa _1 ) =  q \nu_3 \nu_4 \nu_7 \nu_8 / \kappa_1  , \: s ^2 (\kappa _2 ) = q \kappa_2^3  / (\nu_1 \nu_2 \nu_5 \nu_6 ) ,  \nonumber \\
& s^2 (f) = \frac{q \nu_3 \nu_4 \nu_7 \nu_8 }{ \kappa_1  f} \frac{(g -\nu _5 /\kappa _2 )(g  - \nu _6 /\kappa _2)}{( g -1/\nu_1 )( g -1 /\nu_2 )}  , \nonumber \\
& s^2 (g)  = \frac{1}{g} \frac{1}{s^2 (\kappa _2)}\frac{( 1/s^2 (f) -1/\nu_7 )(1/s^2 (f) -1 /\nu_8 )}{(1/s^2 (f) -s^2 (1 /\nu _4 ))(1/s^2 (f) -s^2 (1 /\nu _3 ))} . \nonumber 
\end{align}
Hence, the action of the operator $(\pi _2 \pi _1 s_2 s_1 s_0 s_2  )^2 $ does not coincide with the time evolution in equation (\ref{eq:qPD15}), but the two expressions resemble each other.

The time evolution of the $q$-Painlev\'e equation of type $D^{(1)}_5 $ is expressed as the following theorem,
\begin{thm}  $($\cite[Theorem 2.1 (i)]{TsqP}$)$ \label{thm:D5TW} 
Let $\Xi $ be the transformation of the parameters defined by
\begin{align}
& (\nu_1 , \nu_2 , \nu_3 , \nu_4 ,\nu_5 ,\nu_6 , \nu_7 ,\nu_8 ;\kappa_1 ,\kappa_2 ; f ,g) \label{eq:D5Xi} \\
&  \mapsto \Bigl( \frac{\nu _1 \nu_ 5 \nu_6 }{ \kappa_2 } , \frac{\nu _2 \nu_ 5 \nu_6}{\kappa_2} , \frac{\kappa_1 }{q \nu_4 } , \frac{\kappa_1 }{q \nu_3 } ,\frac{\nu _5 \nu_ 1 \nu_2 }{\kappa_2 } ,\frac{\nu _6 \nu_ 1 \nu_2 }{\kappa_2 } , \frac{\kappa_1 }{q \nu_8 } , \frac{\kappa_1 }{q \nu_7 }  ; \nonumber\\
& \qquad \qquad \qquad \qquad \frac{\kappa_1^3}{q^2 \nu_3 \nu_4 \nu_7 \nu_8 } ,\frac{\nu_1 \nu_2 \nu_5 \nu_6 }{\kappa _2 } ; \frac{f \kappa_1}{q \nu_3 \nu_4 } , \frac{g \kappa_2 }{\nu_ 5 \nu_6 } \Bigr) .  \nonumber
\end{align}
Let $T$ be the transformation of the parameters defined by 
\begin{equation}
T= \Xi \cdot (\pi _2 \pi _1 s_2 s_1 s_0 s_2 )^2 .
\end{equation}
Then 
\begin{align}
& T(f)= \frac{ \nu_3 \nu_4  }{ f  } \frac{(g - \nu _5 /\kappa _2 )(g  - \nu _6 /\kappa _2)}{( g -1/\nu_1 )( g -1 /\nu_2 )}, \label{eq:D5TfTg} \\
&  T(g)=  \frac{1}{g \nu_1 \nu_2 }\frac{( T(f) - T(\kappa_1 /\nu_7)) ( T(f) - T(\kappa_1/ \nu_8))}{(T(f) - \nu_3 )(T(f) - \nu_4)}. \nonumber
\end{align}
Namely the operator $T$ represents the time evolution of $q$-$P(D^{(1)}_5)$ (see equation (\ref{eq:qPD15})).
On the other parameters, we have
\begin{align}
& \displaystyle T(\nu _i)= \nu _i \; (i=1,2,\dots ,8) , \;  T(\kappa _1 ) =  \kappa _1  / q , \; T(\kappa _2 ) = q \kappa _2  . \label{eq:D5T}
\end{align}
\end{thm}

\subsection{} \label{sec:WE8}
On the $q$-Painlev\'e equation of type $E^{(1)}_{8} $ ($q$-$P(E^{(1)}_{8})$), we recall the expression in Yamada's paper \cite{Y}, which is written as
\begin{align}
& \frac{(\overline{f}-g)(f-g)-(\kappa _1/q-\kappa _2)(\kappa _1-\kappa _2)/ \kappa _2}
{( q \overline{f}/ \kappa _1-g/\kappa _2)(f/\kappa _1-g/\kappa _2)
-(q/\kappa _1-1/\kappa _2)(1/\kappa _1-1/\kappa _2)\kappa _2}
=\frac{\kappa _1^2\kappa _2^4}{q}\dfrac{P_n(\kappa _2,g)}{P_d(\kappa _2,g)}, \label{eq:qPE18} \\
& \frac{(\overline{f}-\overline{g})(\overline{f}-g)-(\kappa _1/q- q \kappa _2 )(\kappa _1/q-\kappa _2)q/\kappa _1}
{(q \overline{f}/\kappa _1-\overline{g}/(q \kappa _2))
(q \overline{f}/\kappa _1-g/\kappa _2)-(q/\kappa _1- 1/(q \kappa _2))(q/\kappa _1-1/\kappa _2)\kappa _1/q} \nonumber \\
& \qquad \qquad \qquad =\frac{\kappa _1^4\kappa _2^2}{q^3}\dfrac{P_n(\kappa _1/q,\overline{f})}{P_d(\kappa _1/q ,\overline{f})}, \nonumber 
\end{align}
where
\begin{align}
& P_n(h,x)= x^4-m_1 x^3+(m_2-3h -h^{-3}m_8) x^2 \\
& \qquad +(2hm_1-m_3+h^{-2}m_7) x +(h^2 - h m_2+m_4-h^{-1}m_6+h^{-2}m_8),  \nonumber \\
& P_d(h,x)= m_8 x^4-hm_7 x^3+(h^2m_6-3hm_8 - h^{5} ) x^2  \nonumber \\
& \qquad +(2h^2m_7-h^3m_5+h^{5}m_1) x +(h^6 -h^5m_2+h^4m_4-h^{3}m_6+h^{2}m_8),  \nonumber \\
& m_k = \sum _{1\leq i_1 < i_2 < \cdots < i_k \leq 8} \nu _{i_1} \nu _{i_2} \cdots \nu _{i_k} \; (k=1,2,\dots ,8) .  \nonumber 
\end{align}
More informations on the definition of $q$-$P(E^{(1)}_8 )$ are provided in \cite{Y}.
The Weyl group symmetry of the $q$-Painlev\'e equation $q$-$P(E^{(1)}_8 )$ is given by the following action of the operators $s_0, \dots ,s_8$ for the parameters $(\kappa_{1}, \kappa_{2}, \nu_{1}, \dots , \nu_{8}) $ and $(f, g) $ in \cite{Y,KNY}.
\begin{align}
 s_1 : & \: \kappa_1 \leftrightarrow \kappa_2 ,  \: f \leftrightarrow g  ,\label{eq:AffineWeylActionE8} \\
  s_0 : & \: \nu_1 \leftrightarrow \nu_2 , \quad   s_3 : \: \nu_2 \leftrightarrow \nu_3 , \quad  s_4 : \nu_3 \leftrightarrow \nu_4 ,\nonumber \\
  s_5 : & \: \nu_4 \leftrightarrow \nu_5 , \quad   s_6 : \: \nu_5 \leftrightarrow \nu_6 , \quad   s_7 : \: \nu_6 \leftrightarrow \nu_7 , \quad   s_8 : \: \nu_7 \leftrightarrow \nu_8 \nonumber \\
  s_ 2: & \: \nu_1 \rightarrow\frac{\kappa_2}{\nu_2}, \;
  \nu_2 \rightarrow \frac{\kappa_2}{\nu_1}, \;
  \kappa_1 \rightarrow \frac{\kappa_1\kappa_2}{\nu_1\nu_2}, \nonumber \\
& \: f \rightarrow \{ \nu_1 \nu_2 ( \kappa_1 - \kappa_2 ) f g + \kappa_2 ( \nu_1 + \nu_2 ) ( \nu_1 \nu_2 - \kappa_1 ) f \nonumber \\
& \qquad  \qquad - \kappa_1 ( \nu_1 + \nu_2 ) ( \nu_1 \nu_2 - \kappa_2 ) g  + ( \kappa_1 - \kappa_2 ) ( \nu_1 \nu_2 - \kappa_1 ) ( \nu_1 \nu_2 - \kappa_2 ) \} \nonumber \\ 
&  \qquad \quad / \{  \nu_1 \nu_2 ( \nu_1 \nu_2 - \kappa_2 ) f - \nu_1 \nu_2 ( \nu_1 \nu_2 - \kappa_1 ) g - \nu_1 \nu_2 ( \nu_1 + \nu_2 )( \kappa_1 - \kappa_2 ) \} . \nonumber
\end{align}
Then we can confirm that these operations satisfy the relations of the affine Weyl group $W (E^{(1)}_{8})$ whose Dynkin diagram is as follows.\\
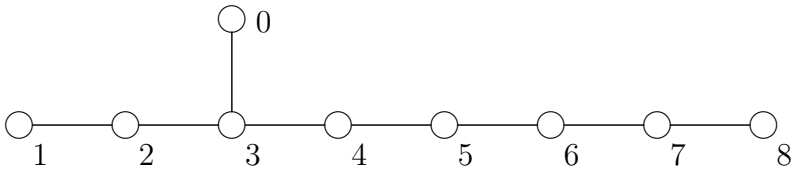
\begin{picture}(0,90)(0,0)
\put(20,20){\circle{10}}
\qbezier(25,20)(40,20)(55,20)
\put(60,20){\circle{10}}
\qbezier(65,20)(80,20)(95,20)
\put(100,20){\circle{10}}
\qbezier(105,20)(120,20)(135,20)
\put(140,20){\circle{10}}
\qbezier(145,20)(160,20)(175,20)
\put(180,20){\circle{10}}
\qbezier(185,20)(200,20)(215,20)
\put(220,20){\circle{10}}
\qbezier(225,20)(240,20)(255,20)
\put(260,20){\circle{10}}
\qbezier(265,20)(280,20)(295,20)
\put(300,20){\circle{10}}
\qbezier(100,25)(100,40)(100,55)
\put(100,60){\circle{10}}
\put(25,5){$1$}
\put(65,5){$2$}
\put(105,5){$3$}
\put(145,5){$4$}
\put(185,5){$5$}
\put(225,5){$6$}
\put(265,5){$7$}
\put(305,5){$8$}
\put(109,55){$0$}
\end{picture}
\\
On $q$-$P(E^{(1)}_8)$, a relationship between the time evolution and the symmetry of $W (E^{(1)}_8) $ was established by Yamada \cite{Y}.
Let $i , j$ be distinct integers in $\{ 1,2,3,4,5,6,7,8 \}$.
Define the action $\mu_ {ij}$ by
\begin{align}
  \mu_ {ij}: & \: \nu_i \rightarrow\frac{\kappa_2}{\nu_j}, \; \nu_j \rightarrow \frac{\kappa_2}{\nu_i}, \; \kappa_1 \rightarrow \frac{\kappa_1\kappa_2}{\nu_i\nu_j}, \\
& \: f \rightarrow \{ \nu_i \nu_j ( \kappa_1 - \kappa_2 ) f g + \kappa_2 ( \nu_i + \nu_j ) ( \nu_i \nu_j - \kappa_1 ) f  \nonumber \\
& \qquad \qquad - \kappa_1 ( \nu_i + \nu_j ) ( \nu_i \nu_j - \kappa_2 ) g + ( \kappa_1 - \kappa_2 ) ( \nu_i \nu_j - \kappa_1 ) ( \nu_i \nu_j - \kappa_2 ) \}  \nonumber \\
& \qquad \quad / \{ \nu_i \nu_j ( \nu_i \nu_j - \kappa_2 ) f - \nu_i \nu_j ( \nu_i \nu_j - \kappa_1 ) g - \nu_i \nu_j ( \nu_i + \nu_j )( \kappa_1 - \kappa_2 ) \} .  \nonumber 
\end{align}
Let $r$ and $T_1$ be the following compositions
\begin{equation}
r=s_{0}\mu_{12}s_{4}\mu_{34}s_{6}\mu_{56}s_{8}\mu_{78}, \quad T_1= s_1 r s_1 r. \label{eq:rT1def} 
\end{equation}
Their actions were calculated in \cite{Y} as
\begin{align}
& r(\kappa _1)=v \kappa _2, \; r(\kappa _2)=\kappa _2, \; r(\nu _i)=\dfrac{\kappa _2}{\nu _i}, \;  r(f)=\overline{f} v, \;  r(g)=g, \label{eq:rT1act} \\
& T_1(\kappa _1)=\frac{\kappa _1}{q} v^2, \; T_1(\kappa _2)=q \kappa _2 v^2, \; T_1(\nu _i) = \nu _i v, \; T_1(f)=\overline{f} v, \; T_1(g)=\overline{g} v, \nonumber 
\end{align}
where $v=q\kappa _2/\kappa _1$ and  $\overline{\: \cdot \: }$ is the time evolution defined in equation (\ref{eq:qPE18}).
Hence, the affine Weyl group translation $T_1$ is related to the time evolution $\overline{\: \cdot \: }$ by re-scaling  the variables $\nu_{1}, \nu_{2}, \dots , \nu_{8} , \kappa _1 , \kappa _2, f,g $.
Note that the action of $T_1$ can be written as a composition of the operators $s_0, s_1 , \dots , s_8$, which follows from $\mu _{12} = s_2$, $\mu _{34} = s_{3} s_{4} s_{0} s_{3} \mu _{12} s_{3} s_{4} s_{0} s_{3}$, $\mu _{56} = s_{5} s_{6} s_{4} s_{5} \mu _{34} s_{5} s_{6} s_{4} s_{5}$ and $\mu _{78} = s_{7} s_{8} s_{6} s_{7} \mu _{56} s_{7} s_{8} s_{6} s_{7}$.
By introducing the operator $\Xi = T \cdot T_1^{-1}$, we obtain the following theorem.
\begin{thm} $($c.f.~\cite[Appendix A]{Y}$)$  \label{thm:E8TW}
Let $\Xi $ be the transformation of the parameters defined by
\begin{align}
& (\nu_1 , \nu_2 , \nu_3 , \nu_4 ,\nu_5 ,\nu_6 , \nu_7 ,\nu_8 ;\kappa_1 ,\kappa_2 ; f ,g) \label{eq:E8Xi} \\
& \mapsto  \Bigl( \frac{\kappa_1  }{ q \kappa_2 } \nu _1 ,\frac{\kappa_1  }{ q \kappa_2 } \nu _2 ,\frac{\kappa_1  }{ q \kappa_2 } \nu _3 ,\frac{\kappa_1  }{ q \kappa_2 } \nu _4 ,\frac{\kappa_1  }{ q \kappa_2 } \nu _5 ,\frac{\kappa_1  }{ q \kappa_2 } \nu _6 ,\frac{\kappa_1  }{ q \kappa_2 } \nu _7 ,\frac{\kappa_1  }{ q \kappa_2 } \nu _8  ; \nonumber \\
& \qquad \frac{\kappa_1^2}{q^2 \kappa _2^2 } \kappa_1,\frac{\kappa_1^2 }{q^2 \kappa ^2_2 } \kappa _2; \frac{\kappa_1}{q \kappa _2} f , \frac{\kappa_1}{q \kappa _2} g \Bigr) ,\nonumber
\end{align}
Let $T$ be the transformation of the parameters defined by 
\begin{equation}
T= \Xi \cdot (  s_1 s_{0} s_{2} s_{4} \mu_{34} s_{6} \mu_{56} s_{8} \mu_{78}   )^2 , \label{eq:E8TXis}
\end{equation}
where $\mu _{34} = s_{3} s_{4} s_{0} s_{3} s_{2} s_{3} s_{4} s_{0} s_{3}$, $\mu _{56} = s_{5} s_{6} s_{4} s_{5} \mu _{34} s_{5} s_{6} s_{4} s_{5}$ and $\mu _{78} = s_{7} s_{8} s_{6} s_{7} \mu _{56} s_{7} s_{8} s_{6} s_{7}$.
Then, the operator $T$ represents the time evolution of $q$-$P(E^{(1)}_8)$ (see equation (\ref{eq:qPE18})).
On the other parameters, we have
\begin{align}
& \displaystyle T(\nu _i)= \nu _i \; (i=1,2,\dots ,8) , \;  T(\kappa _1 ) =  \kappa _1  / q , \; T(\kappa _2 ) = q \kappa _2  . \label{eq:E8T} 
\end{align}
\end{thm}

\section{Equivalence in representations and $q$-Painlev\'e equations} \label{sec:equiv}

In order to investigate the representations of the extended affine Weyl groups in Section \ref{sec:prel}, we introduce the equivalences on the variables in the representations.
The equivalence is an evidence of the reducibility of the representations.
The difference between the time evolution of the $q$-Painlev\'e equation and the action of the parallel translation of the corresponding Weyl group in Sections \ref{sec:WE6}--\ref{sec:WE8} can be interpreted by the equivalence.

\subsection{The case $E^{(1)}_6 $} \label{sec:E16}

In connection with the representation of the extended affine Weyl group $\widetilde{W}(E^{(1)}_{6})$ given in equation (\ref{eq:AffineWeylActionE6}), we introduce an equivalence of the parameters as follows.
\begin{align}
&  (\nu '_1 , \nu '_2 , \nu '_3 , \nu '_4 ,\nu '_5 ,\nu '_6 , \nu '_7 ,\nu '_8 ; \kappa '_1 ,\kappa '_2 ; f' ,g' ) \sim  (\nu_1 , \nu_2 , \nu_3 , \nu_4 ,\nu_5 ,\nu_6 , \nu_7 ,\nu_8 ;\kappa_1 ,\kappa_2 ; f ,g ) \label{eq:equivE6} \\
& \Leftrightarrow \mbox{There exist } a, b, c \in \C (\nu_1 , \nu_2 , \cdots ,\nu_8 ,\kappa_1 ,\kappa_2 , f ,g ) \setminus \{ 0 \} \mbox{ such that }  \nonumber \\
& \qquad (\nu '_1 , \nu '_2 , \nu '_3 , \nu '_4 ,\nu '_5 ,\nu '_6 , \nu '_7 ,\nu '_8 ; \kappa '_1 ,\kappa '_2 ; f' ,g' )  \nonumber \\
& \qquad \qquad = (a \nu_1 , a \nu_2 , a \nu_3 , a \nu_4 ,b \nu_5 , b \nu_6 , c \nu_7 , c \nu_8 ; ac \kappa_1 , ab \kappa_2 ; af ,g/a) .  \nonumber
\end{align}
An explanation why the equivalence is defined as above is given in the appendix, that is related to a symmetry of the linear $q$-difference equation corresponding to the $q$-Painlev\'e equation $q$-$P(E^{(1)}_6)$.

We have the following statement on the representation of $\widetilde{W}(E^{(1)}_{6})$ in equation (\ref{eq:AffineWeylActionE6}) and the equivalence.
\begin{thm} \label{thm:E6compati}
Let $(\nu '_1 , \nu '_2 , \nu '_3 , \nu '_4 ,\nu '_5 ,\nu '_6 , \nu '_7 ,\nu '_8 ; \kappa '_1 ,\kappa '_2 ; f' ,g' ) $ be a tuple of the parameters which is equivalent to $ (\nu_1 , \nu_2 , \nu_3 , \nu_4 ,\nu_5 ,\nu_6 , \nu_7 ,\nu_8 ;\kappa_1 ,\kappa_2 ; f ,g )$.
The actions of the generators of  $\widetilde{W}(E^{(1)}_6)$ in equation (\ref{eq:AffineWeylActionE6}) to $(\nu '_1 , \nu '_2 , \nu '_3 , \nu '_4 ,\nu '_5 ,\nu '_6 , \nu '_7 ,\nu '_8 ; \kappa '_1 ,\kappa '_2 ; f' ,g' ) $ are written as the same form of the actions to $(\nu _1 , \nu _2 , \nu _3 , \nu _4 ,\nu _5 ,\nu _6 , \nu _7 ,\nu _8 ; \kappa _1 ,\kappa _2 ; f ,g ) $ up to the equivalence.
\end{thm}
\begin{proof}
By the assumption, there exist $ a, b, c \in \C (\nu_1 , \nu_2 , \cdots ,\nu_8 ,\kappa_1 ,\kappa_2 , f ,g ) \setminus \{ 0 \} $ such that 
\begin{align}
&  (\nu '_1 , \nu '_2 , \nu '_3 , \nu '_4 ,\nu '_5 ,\nu '_6 , \nu '_7 ,\nu '_8 ; \kappa '_1 ,\kappa '_2 ; f' ,g' )  \label{eq:equivE6'} \\
& = (a \nu_1 , a \nu_2 , a \nu_3 , a \nu_4 ,b \nu_5 , b \nu_6 , c \nu_7 , c \nu_8 ; ac \kappa_1 , ab \kappa_2 ; af ,g/a). \nonumber
\end{align}
We confirm the theorem on the action of $\pi _1$.
Recall that
\begin{align}
& (\pi _1 ( \nu _1) , \pi _1 (\nu _2) , \pi _1 ( \nu _3) , \pi _1 ( \nu _4) , \pi _1 (\nu _5) , \pi _1 (\nu _6) , \pi _1 ( \nu _7) , \pi _1 (\nu _8) ;  \pi _1 (\kappa _1) , \pi _1 (\kappa _2) ;  \pi _1 (f ) , \pi _1 (g) ) \\
& = ( \frac{\nu  _2 }{\kappa _2}  , \frac{\nu  _1 }{\kappa _2}   , \frac{1}{ \nu _6} , \frac{1}{\nu  _5}  , \frac{1}{\nu  _4}  ,  \frac{1}{\nu  _3}  ,  \frac{1}{\nu _7} , \frac{1}{\nu _8} ; \frac{\nu _1 \nu _2 }{ \kappa_1\kappa _2}  , \frac{1}{ \kappa _2} ; \frac{\nu _1 \nu_2 ( 1- fg )}{ \kappa_2 \{ \nu _1 \nu_2 g + f- (\nu _1 + \nu _2 ) fg \} } , \kappa_2 g  ). \nonumber 
\end{align}
We calculate the action of $\pi _1$ to $\nu '_1 , \nu '_2 , \dots  , f' ,g' $ which satisfy equation (\ref{eq:equivE6'}).
We have
\begin{align}
&  \pi _1 ( \nu '_5 ) = \pi _1 ( b ) \pi _1 (\nu _5 )  =  \frac{\pi _1 ( b ) }{\nu _4} = \frac{\pi _1 ( b ) a }{ \nu '_4 } , \; \pi _1 ( f ') =  \pi _1 ( a ) \pi _1 (f )  \\
&  \quad = \frac{ \pi _1 ( a ) \nu _1 \nu_2 ( 1- fg )}{ \kappa_2 \{ \nu _1 \nu_2 g + f- (\nu _1 + \nu _2 ) fg \} }  = \frac{ \pi _1 ( a ) b \nu '_1 \nu '_2 ( 1- f'g' )}{ \kappa ' _2 \{ \nu '_1 \nu'_2 g' + f' - (\nu '_1 + \nu '_2 ) f' g' \} }  , \nonumber
\end{align}
and so on.
It follows from a direct calculation that
\begin{align}
& (\pi _1 ( \nu '_1) , \pi _1 (\nu '_2) , \pi _1 ( \nu '_3) , \pi _1 ( \nu '_4) , \dots , \pi _1 (\nu '_8) ; \pi _1 (\kappa '_1) , \pi _1 (\kappa '_2) ;  \pi _1 (f' ) , \pi _1 (g') ) \\
& = ( \pi _1 ( a ) b \frac{ \nu '_2 }{ \kappa ' _2} , \pi _1 ( a ) b \frac{ \nu '_1 }{\kappa ' _2} ,  \frac{ \pi _1 ( a ) b  }{ \nu '_6 }  , \frac{ \pi _1 ( a ) b  }{ \nu '_5 } , \frac{  \pi _1 ( b ) a }{ \nu '_4 } , \frac{ \pi _1 ( b ) a }{ \nu '_3} , \frac{ \pi _1 (c ) c }{ \nu '_7} , \frac{ \pi _1 ( c ) c }{ \nu '_8} ; \nonumber \\
& \qquad \pi _1 ( a c) b c  \frac{  \nu '_1 \nu '_2 }{ \kappa '_1\kappa '_2}  , \frac{ \pi _1 ( ab ) ab }{\kappa '_2 }  ; \pi _1 ( a ) b  \frac{\nu '_1 \nu '_2 ( 1- f'g' )}{ \kappa ' _2 \{ \nu '_1 \nu'_2 g' + f' - (\nu '_1 + \nu '_2 ) f' g' \} } , \frac{\kappa '_2 g' }{\pi _1 ( a ) b }  ) , \nonumber 
\end{align} 
which is equivalent to
\begin{align}
& ( \frac{ \nu '_2 }{ \kappa ' _2} , \frac{ \nu '_1 }{\kappa ' _2} ,  \frac{1 }{ \nu '_6 }  , \frac{ 1 }{ \nu '_5 } , \frac{1 }{ \nu '_4 } , \frac{ 1}{ \nu '_3} , \frac{ 1 }{ \nu '_7} , \frac{ 1 }{ \nu '_8} ;  \frac{  \nu '_1 \nu '_2 }{ \kappa '_1\kappa '_2}  , \frac{ 1 }{\kappa '_2 }  ;   \frac{\nu '_1 \nu '_2 ( 1- f'g' )}{ \kappa ' _2 \{ \nu '_1 \nu'_2 g' + f' - (\nu '_1 + \nu '_2 ) f' g' \} } , \kappa '_2 g'  ) ,
\end{align} 
where the parameters $\pi _1 ( a ) b  ,  \pi _1 ( b ) a ,  \pi _1 ( c ) c$ correspond to the parameters $a,b,c $ in equation (\ref{eq:equivE6}).
Therefore, the action of $\pi _1 $ to $(\nu '_1 , \nu '_2 ,$ $ \nu '_3 , \nu '_4 ,\nu '_5 ,\nu '_6 , \nu '_7 ,\nu '_8 ; \kappa '_1 ,\kappa '_2 ; f' ,g' ) $ is written as the same form of the action to $(\nu _1 , \nu _2 , \nu _3 ,$ $ \nu _4 ,\nu _5 ,\nu _6 , \nu _7 ,\nu _8 ; \kappa _1 ,\kappa _2 ; f ,g ) $ up to the equivalence.
It is shown similarly that the action of the other generators of $\widetilde{W}(E^{(1)}_{6})$ to $(\nu '_1 , \nu '_2 ,$ $ \nu '_3 , \nu '_4 ,\nu '_5 ,\nu '_6 , \nu '_7 ,\nu '_8 ; \kappa '_1 ,\kappa '_2 ; f' ,g' ) $ is written as the same form of the action to $(\nu _1 , \nu _2 , \nu _3 ,$ $ \nu _4 ,\nu _5 ,\nu _6 , \nu _7 ,\nu _8 ; \kappa _1 ,\kappa _2 ; f ,g ) $ up to the equivalence.
We note crucial equations on the actions of some generators of $\widetilde{W}(E^{(1)}_{6})$.
\begin{align}
& (\pi _2 ( \nu _1) , \pi _2 (\nu _2) , \pi _2 ( \nu _3) , \pi _2 ( \nu _4) , \dots , \pi _2 (\nu _8) ;  \pi _2 (\kappa _1) , \pi _2 (\kappa _2) ;  \pi _2 (f ) , \pi _2 (g) ) \\
& = (\frac{1}{\nu_1}, \frac{1}{\nu_2}, \frac{1}{\nu_3}, \frac{1}{\nu_4}, \frac{1}{\nu_8} , \frac{1}{\nu_7}, \frac{1}{\nu_6}, \frac{1}{\nu_5} ; \frac{1}{\kappa_2} , \frac{1}{\kappa_1}  ; g , f ),  \nonumber \\
& (\pi _2 ( \nu '_1) , \pi _2 (\nu '_2) , \pi _2 ( \nu '_3) , \pi _2 ( \nu '_4) , \dots , \pi _2 (\nu '_8) ;  \pi _2 (\kappa '_1) , \pi _2 (\kappa '_2) ;  \pi _2 (f' ) , \pi _2 (g') ) \nonumber \\
& =  (\frac{\pi _2 ( a ) a }{\nu ' _1}, \frac{\pi _2 ( a ) a }{\nu  '_2}, \frac{\pi _2 ( a ) a }{\nu  '_3}, \frac{\pi _2 ( a ) a }{\nu  '_4}, \frac{\pi _2 ( b ) c }{\nu  '_8} , \frac{\pi _2 ( b ) c }{\nu  '_7}, \frac{\pi _2 ( c ) b }{\nu  '_6}, \frac{\pi _2 ( c ) b }{\nu  '_5} ; \nonumber \\
& \qquad   \frac{\pi _2 ( a c ) a b }{\kappa  '_2} , \frac{\pi _2 ( a b ) a c }{\kappa  '_1 } ; \pi _2 ( a ) a g ' ,  \frac{f ' }{\pi _2 ( a ) a}) . \nonumber
\end{align} 
\begin{align}
& (s _2 ( \nu _1) , s _2 (\nu _2) , s _2 ( \nu _3) , s _2 ( \nu _4) , \dots , s _2 (\nu _8) ;  s _2 (\kappa _1) , s _2 (\kappa _2) ;  s _2 (f ) , s _2 (g) ) \\
& =  (\frac{\kappa_2 }{\nu_6 }, \nu_2 , \nu_3 , \nu_4 ,\nu_5 ,\frac{\kappa_2 }{\nu_1} , \nu_7 ,\nu_8 ; \frac{\kappa_1 \kappa_2 }{\nu_1 \nu_6} ,\kappa_2 ;  f\, \frac{\kappa_2 ( \nu_1 g-1)  }{ (\nu _1 \nu_6 - \kappa_2  )fg + \nu _1 \kappa_2 g - \nu _1 \nu_6 } ,g ) , \nonumber \\
& (s _2 ( \nu '_1) , s _2 (\nu '_2) , s _2 ( \nu '_3) , s _2 ( \nu '_4) , \dots , s _2 (\nu '_8) ;  s _2 (\kappa '_1) , s _2 (\kappa '_2) ; s _2 (f' ) , s _2 (g') ) \nonumber \\
& =  (\frac{s _2 (a)}{a} \frac{\kappa '_2 }{\nu '_6} , \frac{s _2 (a)}{a} \nu '_2 , \frac{s _2 (a)}{a} \nu '_3 , \frac{s _2 (a)}{a} \nu '_4 , \frac{s _2 (b)}{b} \nu '_5 , \frac{s _2 (b)}{b} \frac{\kappa '_2 }{\nu '_1} , \frac{s _2 (c)}{c} \nu '_7 , \frac{s _2 (c)}{c} \nu '_8 ; \nonumber \\
& \qquad \frac{s _2 (ac)}{ac} \frac{\kappa '_1 \kappa '_2 }{\nu '_1 \nu '_6} , \frac{s _2 (ab)}{ab} \kappa '_2 ; \frac{s _2 (a)}{a} f' \frac{\kappa '_2 ( \nu '_1 g' -1)}{(\nu ' _1 \nu '_6 - \kappa '_2  ) f' g' + \nu ' _1 \kappa '_2 g' - \nu ' _1 \nu '_6 } , \frac{a}{s _2 (a)} g' ) . \nonumber 
\end{align} 
\begin{align}
& (s _6 ( \nu _1) , s _6 (\nu _2) , s _6 ( \nu _3) , s _6 ( \nu _4) , \dots , s _6 (\nu _8) ;  s _6 (\kappa _1) , s _6 (\kappa _2) ;  s _6 (f ) , s _6 (g) ) \\
& =  (\frac{\kappa_1 }{\nu_7} , \nu_2 , \nu_3 , \nu_4 ,\nu_5 ,\nu_6 , \frac{\kappa_1 }{\nu_1} ,\nu_8 ; \kappa_1, \frac{\kappa_1 \kappa_2 }{\nu_1 \nu _7 } ;  f, \frac{\nu _7 ( \nu_1 - f ) g }{ \kappa_1  -  \nu _7 f + ( \nu _1 \nu_7 - \kappa_1 )fg } ) , \nonumber \\
& (s _6 ( \nu '_1) , s _6 (\nu '_2) , s _6 ( \nu '_3) , s _6 ( \nu '_4) , \dots , s _6 (\nu '_8) ; s _6 (\kappa '_1) , s _6 (\kappa '_2) ; s _6 (f' ) , s _6 (g') ) \nonumber \\& =  (\frac{s _6 (a)}{a} \frac{\kappa '_1 }{\nu '_7} , \frac{s _6 (a)}{a} \nu '_2 , \frac{s _6 (a)}{a} \nu '_3 , \frac{s _6 (a)}{a} \nu '_4 , \frac{s _6 (b)}{b} \nu '_5  , \frac{s _6 (b)}{b} \nu '_6 , \frac{s _6 (c)}{c} \frac{\kappa '_1 }{\nu '_1} , \frac{s _6 (c)}{c} \nu '_8 ; \nonumber \\
& \qquad \frac{s _6 (ac)}{ac} \kappa '_1, \frac{s _6 (ab)}{ab} \frac{\kappa '_1 \kappa '_2 }{\nu '_1 \nu '_7 } ;  \frac{s _6 (a)}{a} f', \frac{a}{s _6 (a)} \frac{\nu '_7 ( \nu '_1 - f' ) g' }{ \kappa '_1  -  \nu '_7 f' + ( \nu '_1 \nu '_7 - \kappa '_1 )f'g' } ) . \nonumber 
\end{align} 
\end{proof}
The operator $\Xi $, whose action is described in equations (\ref{eq:E6Xinu}) and (\ref{eq:E6Xifg}), is related to the definition of the equivalence in equation (\ref{eq:equivE6}).
Namely, we have
\begin{align}
& ( \Xi (\nu_1) , \Xi (\nu_2 ) , \Xi (\nu_3 ) , \Xi (\nu_4 ) , \Xi (\nu_5 ) , \Xi (\nu_6 ) , \Xi (\nu_7 ) , \Xi (\nu_8 ) ; \Xi (\kappa_1) , \Xi (\kappa_2 ) ; \Xi (f ) , \Xi (g ) )  \\ 
& \sim (\nu_1 , \nu_2 , \nu_3 , \nu_4 ,\nu_5 ,\nu_6 , \nu_7 ,\nu_8 ;\kappa_1 ,\kappa_2 ; f ,g ) , \nonumber
\end{align}
where the parameters $a,b,c$ in equation (\ref{eq:equivE6}) is written as 
\begin{align}
& a = \frac{ \nu_ 5 \nu_ 6 }{\kappa_2 } , \; b = \frac{\kappa_1 }{q \kappa_2 } , \; c = \frac{\kappa_1 }{q \nu_7 \nu_8 } 
\end{align}
for this case.
We obtain the following theorem by reformulating the setting of the operators $\Xi $ and $T$.
\begin{thm}
Let $\Xi $ be the transformation of the parameters defined by
\begin{align}
\Xi  & : (\nu_1 , \nu_2 , \nu_3 , \nu_4 ,\nu_5 ,\nu_6 , \nu_7 ,\nu_8 ;\kappa_1 ,\kappa_2 ; f ,g) \label{eq:E6Xi} \\
& \mapsto \; \Bigl( \frac{ \nu_ 5 \nu_ 6 }{\kappa_2 } \nu _1 , \frac{ \nu_ 5 \nu_ 6 }{\kappa_2 }  \nu _2 , \frac{ \nu_ 5 \nu_ 6 }{\kappa_2 }  \nu _3 , \frac{ \nu_ 5 \nu_ 6 }{\kappa_2 }  \nu _4 , \frac{\kappa_1 }{q \kappa_2 } \nu _5,\frac{\kappa_1 }{q \kappa_2 } \nu _6 , \nonumber \\
& \qquad \frac{\kappa_1 }{q \nu _7 \nu_8 } \nu _7, \frac{\kappa_1 }{q \nu_7 \nu _8} \nu _8 ;  \frac{\kappa _1 \nu_5 \nu_6 }{q \kappa_2 \nu_7 \nu_8 } \kappa _1 ,\frac{\kappa_1 \nu_5 \nu_6 }{q \kappa_2^2} \kappa _2   ; \frac{ \nu_ 5 \nu_ 6 }{\kappa_2 } f, \frac{\kappa_2 }{ \nu_ 5 \nu_ 6 } g \Bigr) \nonumber
\end{align}
and set $T= \Xi \cdot (\pi _1 \pi _2 s_4 s_5 s_3 s_6 s_4 s_3 s_0 s_6 )^2 $.
Then, the operator $T$ coincides with the time evolution of $q$-$P(E^{(1)}_6)$ given in equation (\ref{eq:qPE16}).
The transformation $\Xi $ is within the equivalence of the parameters in equation (\ref{eq:equivE6}).
\end{thm}
Note that the expression of $\Xi $ in \cite{TsqP} for the case $E^{(1)}_6 $ contains a mistake, and the transformation $\Xi $ in \cite[Theorem 3.1 (i)]{TsqP} should be replaced by equation (\ref{eq:E6Xi}).

The number of the parameters of Kajiwara-Noumi-Yamada realization is too much, and it may cause the difference between the time evolution and the Weyl group symmetry.
By the equivalence of the parameters, we may regard that the number of the parameters are deduced.
The difference between the time evolution and the Weyl group symmetry can be understood by finding the transformation $\Xi $, which is related to the equivalence.

\subsection{The case $E^{(1)}_7 $} \label{sec:E17}

In connection with the representation of the extended affine Weyl group $\widetilde{W}(E^{(1)}_{7})$ given in equation (\ref{eq:AffineWeylActionE7}), we introduce an equivalence of the parameters as follows.
\begin{align}
&  (\nu '_1 , \nu '_2 , \nu '_3 , \nu '_4 ,\nu '_5 ,\nu '_6 , \nu '_7 ,\nu '_8 ; \kappa '_1 ,\kappa '_2 ; f' ,g' ) \sim  (\nu_1 , \nu_2 , \nu_3 , \nu_4 ,\nu_5 ,\nu_6 , \nu_7 ,\nu_8 ;\kappa_1 ,\kappa_2 ; f ,g ) \label{eq:equivE7} \\
& \Leftrightarrow \mbox{There exist } a, b \in \C (\nu_1 , \nu_2 , \cdots ,\nu_8 ,\kappa_1 ,\kappa_2 , f ,g ) \setminus \{ 0 \} \mbox{ such that }  \nonumber \\
& \qquad (\nu '_1 , \nu '_2 , \nu '_3 , \nu '_4 ,\nu '_5 ,\nu '_6 , \nu '_7 ,\nu '_8 ; \kappa '_1 ,\kappa '_2 ; f' ,g' )  \nonumber \\
& \qquad \qquad = (a \nu_1 , a \nu_2 , a \nu_3 , a \nu_4 ,b \nu_5 , b \nu_6 , b \nu_7 , b \nu_8 ; ab \kappa_1 , ab \kappa_2 ; af ,g/a) .  \nonumber
\end{align}
Note that the definition of the equivalence originates from \cite[(4.7)]{TsqP}.
We have the following statement on the representation of $\widetilde{W}(E^{(1)}_{7})$ in equation (\ref{eq:AffineWeylActionE7}) and the equivalence.
\begin{thm}
Let $(\nu '_1 , \nu '_2 , \nu '_3 , \nu '_4 ,\nu '_5 ,\nu '_6 , \nu '_7 ,\nu '_8 ; \kappa '_1 ,\kappa '_2 ; f' ,g' ) $ be a tuple of the parameters which is equivalent to $ (\nu_1 , \nu_2 , \nu_3 , \nu_4 ,\nu_5 ,\nu_6 , \nu_7 ,\nu_8 ;\kappa_1 ,\kappa_2 ; f ,g )$.
The actions of the generators of  $\widetilde{W}(E^{(1)}_7)$ to $(\nu '_1 , \nu '_2 , \nu '_3 , \nu '_4 ,\nu '_5 ,\nu '_6 , \nu '_7 ,\nu '_8 ; \kappa '_1 ,\kappa '_2 ; f' ,g' ) $ are written as the same form of the actions to $(\nu _1 , \nu _2 , \nu _3 , \nu _4 ,\nu _5 ,\nu _6 , \nu _7 ,\nu _8 ; \kappa _1 ,\kappa _2 ; f ,g ) $ up to the equivalence.
\end{thm}
\begin{proof}
The theorem is shown similarly to Theorem \ref{thm:E6compati}.
By the assumption, the parameters $ \nu '_1 , \nu '_2 , \nu '_3 , \nu '_4 ,\nu '_5 ,\nu '_6 , \nu '_7 ,\nu '_8 ; \kappa '_1 ,\kappa '_2 ; f' ,g' $ are written as the form of equation (\ref{eq:equivE7}).

We confirm the theorem on the action of $\pi $.
Recall that
\begin{align}
  & ( \pi(\nu_1), \pi(\nu_2), \pi(\nu_3), \pi(\nu_4), \pi(\nu_5), \pi(\nu_6), \pi(\nu_7), \pi(\nu_8); \pi(\kappa_1), \pi(\kappa_2); \pi(f), \pi(g) )\\
  & = ( \frac{1}{\nu_5}, \frac{1}{\nu_6}, \frac{1}{\nu_7}, \frac{1}{\nu_8}, \frac{1}{\nu_1}, \frac{1}{\nu_2}, \frac{1}{\nu_3}, \frac{1}{\nu_4}, \frac{1}{\kappa_1}, \frac{1}{\kappa_2}; \frac{f}{\kappa_1 }, \kappa_2 g ). \nonumber
\end{align}
On the action of $\pi $, we have
\begin{align}
  & ( \pi(\nu'_1), \pi(\nu'_2), \pi(\nu'_3), \pi(\nu'_4), \pi(\nu'_5), \pi(\nu'_6), \pi(\nu'_7), \pi(\nu'_8);\pi(\kappa'_1), \pi(\kappa'_2); \pi(f'), \pi(g') )\\
  & = ( \frac{b \pi(a)}{\nu'_5}, \frac{b \pi(a) }{\nu'_6}, \frac{b \pi(a) }{\nu'_7}, \frac{b \pi(a) }{\nu'_8}, \frac{a \pi(b)}{\nu'_1}, \frac{a \pi(b)}{\nu'_2}, \frac{a \pi(b)}{\nu'_3}, \frac{a \pi(b)}{\nu'_4}; \nonumber \\
  &  \qquad \frac{ab \pi(ab) }{\kappa'_1}, \frac{ab \pi(ab) }{\kappa'_2}; b \pi(a)\frac{ f'}{\kappa'_1},  \frac{\kappa'_2 g' }{b \pi(a)} ) \nonumber \\
  & \sim ( \frac{1}{\nu'_5}, \frac{1}{\nu'_6}, \frac{1}{\nu'_7}, \frac{1}{\nu'_8}, \frac{1}{\nu'_1}, \frac{1}{\nu'_2}, \frac{1}{\nu'_3}, \frac{1}{\nu'_4}; \frac{1}{\kappa'_1}, \frac{1}{\kappa'_2} ; \frac{f'}{\kappa'_1} , \kappa'_2 g' ). \nonumber 
\end{align}
Therefore, the action of $\pi $ to $(\nu '_1 , \nu '_2 , \nu '_3 , \nu '_4 ,\nu '_5 ,\nu '_6 , \nu '_7 ,\nu '_8 ; \kappa '_1 ,\kappa '_2 ; f' ,g' ) $ is written as the same form of the action to $(\nu _1 , \nu _2 , \nu _3 , \nu _4 ,\nu _5 ,\nu _6 , \nu _7 ,\nu _8 ; \kappa _1 ,\kappa _2 ; f ,g ) $ up to the equivalence.
It is shown similarly that the action of the other generators of $\widetilde{W}(E^{(1)}_{7})$ to $(\nu '_1 , \nu '_2 ,$ $ \nu '_3 , \nu '_4 ,\nu '_5 ,\nu '_6 , \nu '_7 ,\nu '_8 ; \kappa '_1 ,\kappa '_2 ; f' ,g' ) $ is written as the same form of the action to $(\nu _1 , \nu _2 , \nu _3 ,$ $ \nu _4 ,\nu _5 ,\nu _6 , \nu _7 ,\nu _8 ; \kappa _1 ,\kappa _2 ; f ,g ) $ up to the equivalence.
We note crucial equations on the actions of some generators of $\widetilde{W}(E^{(1)}_{7})$.
\begin{align}
  & ( s_4 (\nu_1),  s_4 (\nu_2),  s_4 (\nu_3),  s_4 (\nu_4),  s_4 (\nu_5),  s_4 (\nu_6),  s_4 (\nu_7),  s_4 (\nu_8);  s_4 (\kappa_1),  s_4 (\kappa_2);  s_4 (f),  s_4 (g) )\\
  & = (\frac{\kappa_2}{\nu_5} , \nu_2 , \nu_3 , \nu_4 ,\frac{\kappa_2}{\nu_1} ,\nu_6 , \nu_7 ,\nu_8 ; \frac{\kappa_1\kappa_2}{\nu_1\nu_5} ,\kappa_2 ; \nonumber \\
  &  \qquad \frac{- \kappa_2 ( \nu_1 \nu_5 - \kappa_1 )f g - \nu_5 ( \kappa_1 - \kappa_2 )f + \kappa_1 ( \nu_1 \nu_5 - \kappa_2 )}{\nu _5 \{ - ( \nu_1 \nu_5 - \kappa_2 )f g + \nu_1 ( \kappa_1 - \kappa_2 )g + ( \nu_1 \nu_5 - \kappa_1 ) \}} ,g ) , \nonumber \\
  & (  s_4 (\nu'_1),  s_4 (\nu'_2),  s_4 (\nu'_3),  s_4 (\nu'_4),  s_4 (\nu'_5),  s_4 (\nu'_6),  s_4 (\nu'_7), s_4 (\nu'_8); s_4 (\kappa'_1), s_4 (\kappa'_2); s_4 (f'), s_4 (g') ) \nonumber \\
  & = ( \frac{s _4 (a)}{a} \frac{\kappa '_2}{\nu '_5} , \frac{s _4 (a)}{a} \nu '_2 , \frac{s _4 (a)}{a} \nu '_3 , \frac{s _4 (a)}{a} \nu '_4 , \frac{s _4 (b)}{b} \frac{\kappa '_2}{\nu '_1} , \frac{s _4 (b)}{b} \nu '_6 , \frac{s _4 (b)}{b} \nu '_7 , \frac{s _4 (b)}{b} \nu '_8 ; \frac{s _4 (ab)}{ab} \frac{\kappa '_1 \kappa '_2}{\nu '_1 \nu '_5} , \nonumber \\
  &  \qquad  \frac{s _4 (ab)}{ab} \kappa '_2 ; \frac{s _4 (a)}{a} \frac{- \kappa '_2 ( \nu '_1 \nu '_5 - \kappa '_1 )f' g' - \nu '_5 ( \kappa '_1 - \kappa '_2 )f' + \kappa '_1 ( \nu '_1 \nu '_5 - \kappa '_2 )}{\nu '_5 \{ - ( \nu '_1 \nu '_5 - \kappa '_2 )f' g' + \nu '_1 ( \kappa '_1 - \kappa '_2 )g' + ( \nu '_1 \nu '_5 - \kappa '_1 ) \}} ,  \frac{a}{s _4 (a)} g' ) . \nonumber 
\end{align}
\begin{align}
  & ( s_0 (\nu_1),  s_0 (\nu_2),  s_0 (\nu_3),  s_0 (\nu_4),  s_0 (\nu_5),  s_0 (\nu_6),  s_0 (\nu_7),  s_0 (\nu_8);  s_0 (\kappa_1),  s_0 (\kappa_2);  s_0 (f),  s_0 (g) )\\
  & = (\nu_1 , \nu_2 , \nu_3 , \nu_4 ,\nu_5 ,\nu_6 , \nu_7 ,\nu_8 ;\kappa_2 ,\kappa_1 ; \frac{1}{g} , \frac{1}{f} ), \nonumber \\
  & ( s_0 (\nu'_1),  s_0 (\nu'_2),  s_0 (\nu'_3),  s_0 (\nu'_4),  s_0 (\nu'_5),  s_0 (\nu'_6),  s_0 (\nu'_7), s_0 (\nu'_8); s_0 (\kappa'_1), s_0 (\kappa'_2); s_0 (f'), s_0 (g') ) \nonumber \\
  &  = ( \frac{s _0 (a)}{a} \nu '_1 , \frac{s _0 (a)}{a} \nu '_2 , \frac{s _0 (a)}{a} \nu '_3 , \frac{s _0 (a)}{a} \nu '_4 , \frac{s _0 (b)}{b} \nu '_5 , \frac{s _0 (b)}{b} \nu '_6 , \frac{s _0 (b)}{b} \nu '_7 , \frac{s _0 (b)}{b} \nu '_8 ; \nonumber \\
  &  \qquad \frac{s _0 (ab)}{ab} \kappa '_2 , \frac{s _0 (ab)}{ab} \kappa '_1 ; \frac{s _0 (a)}{a g'}, \frac{a}{s _0 (a)f'}) . \nonumber
\end{align}
\end{proof}
The difference between the time evolution of $q$-$P(E^{(1)}_7)$ and the particular parallel translation of $\widetilde{W}(E^{(1)}_{7})$ was described by $\Xi $ in equation (\ref{eq:E7Xi}), and we have
\begin{align}
& ( \Xi (\nu_1) , \Xi (\nu_2 ) , \Xi (\nu_3 ) , \Xi (\nu_4 ) , \Xi (\nu_5 ) , \Xi (\nu_6 ) , \Xi (\nu_7 ) , \Xi (\nu_8 ) ; \Xi (\kappa_1) , \Xi (\kappa_2 ) ; \Xi (f ) , \Xi (g ) )  \\ 
& \sim (\nu_1 , \nu_2 , \nu_3 , \nu_4 ,\nu_5 ,\nu_6 , \nu_7 ,\nu_8 ;\kappa_1 ,\kappa_2 ; f ,g ) \nonumber
\end{align}
by the equivalence in equation  (\ref{eq:equivE7}), where the parameters $a,b$ in equation (\ref{eq:equivE7}) is written as $ a = b = \kappa_1 /( q \kappa_2 ) $ for this case.
Namely, we have
\begin{prop}
The transformation $\Xi $ in Theorem \ref{thm:TsqPE7} is within the equivalence of the parameters in equation (\ref{eq:equivE7}).
\end{prop}

\subsection{The case $D^{(1)}_5 $} \label{sec:D15}

In connection with the representation of the extended affine Weyl group $\widetilde{W}(D^{(1)}_{5})$ given in equation (\ref{eq:AffineWeylActionD5}), we introduce an equivalence of the parameters as follows.
\begin{align}
&  (\nu '_1 , \nu '_2 , \nu '_3 , \nu '_4 ,\nu '_5 ,\nu '_6 , \nu '_7 ,\nu '_8 ; \kappa '_1 ,\kappa '_2 ; f' ,g' ) \sim  (\nu_1 , \nu_2 , \nu_3 , \nu_4 ,\nu_5 ,\nu_6 , \nu_7 ,\nu_8 ;\kappa_1 ,\kappa_2 ; f ,g ) \label{eq:equivD5} \\
& \Leftrightarrow \mbox{There exist } a, b, c, d  \in \C (\nu_1 , \nu_2 , \cdots ,\nu_8 ,\kappa_1 ,\kappa_2 , f ,g ) \setminus \{ 0 \} \mbox{ such that }  \nonumber \\
& \qquad (\nu '_1 , \nu '_2 , \nu '_3 , \nu '_4 ,\nu '_5 ,\nu '_6 , \nu '_7 ,\nu '_8 ; \kappa '_1 ,\kappa '_2 ; f' ,g' )  \nonumber \\
& \qquad \qquad =  (a \nu_1 , a \nu_2 , b \nu_3 , b \nu_4 ,c \nu_5 , c \nu_6 , d \nu_7 , d \nu_8 ; bd \kappa_1 , ac \kappa_2 ; bf ,g/a)  .  \nonumber
\end{align}
Note that the definition of the equivalence originates from \cite[(2.10), (2.12)]{TsqP}.
We have the following statement on the representation of $\widetilde{W}(D^{(1)}_{5})$ in equation (\ref{eq:AffineWeylActionD5}) and the equivalence.
\begin{thm}
Let $(\nu '_1 , \nu '_2 , \nu '_3 , \nu '_4 ,\nu '_5 ,\nu '_6 , \nu '_7 ,\nu '_8 ; \kappa '_1 ,\kappa '_2 ; f' ,g' ) $ be a tuple of the parameters which is equivalent to $ (\nu_1 , \nu_2 , \nu_3 , \nu_4 ,\nu_5 ,\nu_6 , \nu_7 ,\nu_8 ;\kappa_1 ,\kappa_2 ; f ,g )$.
The actions of the generators of  $\widetilde{W}(D^{(1)}_{5})$ to $(\nu '_1 , \nu '_2 , \nu '_3 , \nu '_4 ,\nu '_5 ,\nu '_6 , \nu '_7 ,\nu '_8 ; \kappa '_1 ,\kappa '_2 ; f' ,g' ) $ are written as the same form of the actions to $(\nu _1 , \nu _2 , \nu _3 , \nu _4 ,\nu _5 ,\nu _6 , \nu _7 ,\nu _8 ; \kappa _1 ,\kappa _2 ; f ,g ) $ up to the equivalence.
\end{thm}
\begin{proof}
The theorem is shown similarly to Theorem \ref{thm:E6compati}.
By the assumption, the parameters $ \nu '_1 , \nu '_2 , \nu '_3 , \nu '_4 ,\nu '_5 ,\nu '_6 , \nu '_7 ,\nu '_8 ; \kappa '_1 ,\kappa '_2 ; f' ,g' $ are written as the form of equation (\ref{eq:equivD5}).
It is shown that the action of the generators of $\widetilde{W}(D^{(1)}_{5})$ to $(\nu '_1 , \nu '_2 ,$ $ \nu '_3 , \nu '_4 ,\nu '_5 ,\nu '_6 ,$ $ \nu '_7 ,\nu '_8 ; \kappa '_1 ,\kappa '_2 ; f' ,g' ) $ is written as the same form of the action to $(\nu _1 , \nu _2 , \nu _3 ,$ $ \nu _4 ,\nu _5 ,\nu _6 , \nu _7 ,\nu _8 ;$ $ \kappa _1 ,\kappa _2 ; f ,g ) $ up to the equivalence.
We note crucial equations on the actions of some generators of $\widetilde{W}(D^{(1)}_{5})$.
\begin{align}
  & ( \pi _1 (\nu_1), \pi _1 (\nu_2), \pi _1 (\nu_3), \pi _1 (\nu_4), \dots , \pi _1 (\nu_8); \pi _1 (\kappa_1), \pi _1 (\kappa_2); \pi _1 (f), \pi _1 (g) )\\  & = ( \frac{1}{\nu_1},  \frac{1}{\nu_2},  \frac{1}{\nu_7},  \frac{1}{\nu_8}, \frac{1}{\nu_5}, \frac{1}{\nu_6}, \frac{1}{\nu_3}, \frac{1}{\nu_4} ; \frac{1}{\kappa_1},  \frac{1}{\kappa_2}; \frac{f}{\kappa_1},  \frac{1}{g} ), \nonumber \\
  & ( \pi _1 (\nu'_1), \pi _1 (\nu'_2), \pi _1 (\nu'_3), \pi _1 (\nu'_4),  \dots ,  \pi _1 (\nu'_8);\pi _1 (\kappa'_1), \pi _1 (\kappa'_2); \pi _1 (f'), \pi _1 (g') ) \nonumber \\
  & = ( \frac{a \pi _1 (a) }{\nu'_1}, \frac{a \pi _1 (a) }{\nu'_2}, \frac{d \pi _1 (b) }{\nu'_7}, \frac{d \pi _1 (b) }{\nu'_8}, \frac{c \pi _1 (c) }{\nu'_5}, \frac{c \pi _1 (c) }{\nu'_6}, \frac{b \pi _1 (d) }{\nu'_3}, \frac{b \pi _1 (d) }{\nu'_4} ; \nonumber \\
  &  \qquad \frac{bd \pi _1 (bd) }{\kappa'_1}, \frac{ac \pi _1 (ac) }{\kappa'_2}; d \pi _1 (b) \frac{f'}{\kappa'_1}, \frac{1}{a \pi _1 (a)g'}) . \nonumber
\end{align}
\begin{align}
  & ( \pi _2 (\nu_1), \pi _2 (\nu_2), \pi _2 (\nu_3), \pi _2 (\nu_4),  \dots ,  \pi _2 (\nu_8); \pi _2 (\kappa_1), \pi _2 (\kappa_2); \pi _2 (f), \pi _2 (g) )\\
  & = (  \frac{1}{\nu_7},  \frac{1}{\nu_8},  \frac{1}{\nu_5},  \frac{1}{\nu_6},  \frac{1}{\nu_3},  \frac{1}{\nu_4},  \frac{1}{\nu_1}, \frac{1}{\nu_2} ; \frac{1}{\kappa_2}, \frac{1}{\kappa_1}; \frac{1}{\kappa_2 g}, \frac{\kappa_1}{f} ), \nonumber \\
  & ( \pi _2 (\nu'_1), \pi _2 (\nu'_2), \pi _2 (\nu'_3), \pi _2 (\nu'_4),  \dots ,  \pi _2 (\nu'_8);\pi _2 (\kappa'_1), \pi _2 (\kappa'_2); \pi _2 (f'), \pi _2 (g') ) \nonumber \\
  & = ( \frac{d \pi _2 (a) }{\nu '_7}, \frac{d \pi _2 (a) }{\nu '_8}, \frac{c \pi _2 (b) }{\nu '_5}, \frac{c \pi _2 (b) }{\nu '_6}, \frac{b \pi _2 (c) }{\nu '_3}, \frac{b \pi _2 (c) }{\nu '_4}, \frac{a \pi _2 (d) }{\nu '_1}, \frac{a \pi _2 (d) }{\nu '_2}; \nonumber \\
  &  \qquad \frac{ac \pi _2 (bd) }{\kappa '_2}, \frac{bd \pi _2 (ac) }{\kappa '_1}; \frac{c \pi _2 (b) }{\kappa '_2 g'},\frac{1}{ d \pi _2 (a)} \frac{\kappa '_1}{f'} ) . \nonumber
\end{align}
\begin{align}
  & ( s _2 (\nu_1), s _2 (\nu_2), s _2 (\nu_3), s _2 (\nu_4),  \dots ,  s _2 (\nu_8); s _2 (\kappa_1), s _2 (\kappa_2); s _2 (f), s _2 (g) )\\
  & = (\nu_1 , \nu_2 , \frac{\kappa_1}{\nu_7} , \nu_4 ,\nu_5 ,\nu_6 , \frac{\kappa_1}{\nu_3} ,\nu_8 ;\kappa_1 ,\frac{\kappa_1\kappa_2}{\nu_3\nu_7} ; f ,g\, \frac{f - \nu_3}{f - \kappa_1 /\nu_7} ) , \nonumber \\
  & ( s _2 (\nu'_1), s _2 (\nu'_2), s _2 (\nu'_3), s _2 (\nu'_4),  \dots ,  s _2 (\nu'_8); s _2 (\kappa'_1), s _2 (\kappa'_2); s _2 (f'), s _2 (g') ) \nonumber \\
  & = ( \frac{s _2 (a)}{a} \nu '_1 , \frac{s _2 (a)}{a} \nu '_2 , \frac{s _2 (b)}{b} \frac{\kappa '_1}{\nu '_7} , \frac{s _2 (b)}{b} \nu '_4 , \frac{s _2 (c)}{c} \nu '_5 , \frac{s _2 (c)}{c} \nu '_6 , \frac{s _2 (d)}{d} \frac{\kappa '_1}{\nu '_3} , \frac{s _2 (d)}{d} \nu '_8 ; \nonumber \\
  &  \qquad \frac{s _2 (bd)}{bd} \kappa '_1 , \frac{s _2 (ac)}{ac} \frac{\kappa '_1\kappa '_2}{\nu '_3\nu '_7} ; \frac{s _2 (b)}{b} f' , \frac{a}{s _2 (a)} g' \frac{f' - \nu '_3}{f' - \kappa '_1 /\nu '_7} ) . \nonumber 
\end{align}
\begin{align}
  & ( s _3 (\nu_1), s _3 (\nu_2), s _3 (\nu_3), s _3 (\nu_4),  \dots ,  s _3 (\nu_8); s _3 (\kappa_1), s _3 (\kappa_2); s _3 (f), s _3 (g) )\\
  & = (\frac{\kappa_2}{\nu_5} , \nu_2 , \nu_3 , \nu_4 , \frac{\kappa_2}{\nu_1} ,\nu_6 , \nu_7 ,\nu_8 ;\frac{\kappa_1\kappa_2}{\nu_1\nu_5} ,\kappa_2 ; f\, \frac{g - 1/\nu_1}{g - \nu_5/\kappa_2} ,g ) , \nonumber \\ 
  & ( s _3 (\nu'_1), s _3 (\nu'_2), s _3 (\nu'_3), s _3 (\nu'_4),  \dots ,  s _3 (\nu'_8); s _3 (\kappa'_1), s _3 (\kappa'_2); s _3 (f'), s _3 (g') ) \nonumber \\
  & = ( \frac{s _3 (a)}{a} \frac{\kappa '_2}{\nu '_5 }, \frac{s _3 (a)}{a} \nu '_2 , \frac{s _3 (b)}{b} \nu '_3 , \frac{s _3 (b)}{b} \nu '_4 , \frac{s _3 (c)}{c} \frac{\kappa '_2}{\nu '_1} , \frac{s _3 (c)}{c} \nu '_6 , \frac{s _3 (d)}{d} \nu '_7 , \frac{s _3 (d)}{d} \nu '_8 ; \nonumber \\
  &  \qquad \frac{s _3 (bd)}{bd} \frac{\kappa '_1\kappa '_2}{\nu '_1\nu '_5} , \frac{s _3 (ac)}{ac} \kappa '_2 ; \frac{s _3 (b)}{b} f' \frac{g' - 1/\nu '_1}{g' - \nu '_5/\kappa '_2} , \frac{a}{s _3 (a)} g' ) . \nonumber 
\end{align}
\end{proof}

The difference between the time evolution of $q$-$P(D^{(1)}_5)$ and the particular parallel translation of $\widetilde{W}(D^{(1)}_{5})$ was described by $\Xi $ in equation (\ref{eq:D5Xi}), and we have
\begin{align}
& ( \Xi (\nu_1) , \Xi (\nu_2 ) , \Xi (\nu_3 ) , \Xi (\nu_4 ) , \Xi (\nu_5 ) , \Xi (\nu_6 ) , \Xi (\nu_7 ) , \Xi (\nu_8 ) ; \Xi (\kappa_1) , \Xi (\kappa_2 ) ; \Xi (f ) , \Xi (g ) )  \\ 
& \sim (\nu_1 , \nu_2 , \nu_3 , \nu_4 ,\nu_5 ,\nu_6 , \nu_7 ,\nu_8 ;\kappa_1 ,\kappa_2 ; f ,g ) \nonumber
\end{align}
by the equivalence in equation  (\ref{eq:equivD5}).
Namely, we have
\begin{prop}
The transformation $\Xi $ in Theorem \ref{thm:D5TW} is within the equivalence of the parameters in equation (\ref{eq:equivD5}).
\end{prop}

\subsection{The case $E^{(1)}_8 $} \label{sec:E18}
In connection with the representation of the extended affine Weyl group $\widetilde{W}(E^{(1)}_{8})$ given in equation (\ref{eq:AffineWeylActionE8}), we introduce an equivalence of the parameters as follows.
\begin{align}
&  (\nu '_1 , \nu '_2 , \nu '_3 , \nu '_4 ,\nu '_5 ,\nu '_6 , \nu '_7 ,\nu '_8 ; \kappa '_1 ,\kappa '_2 ; f' ,g' ) \sim  (\nu_1 , \nu_2 , \nu_3 , \nu_4 ,\nu_5 ,\nu_6 , \nu_7 ,\nu_8 ;\kappa_1 ,\kappa_2 ; f ,g ) \label{eq:equivE8} \\
& \Leftrightarrow \mbox{There exist } c \in \C (\nu_1 , \nu_2 , \cdots ,\nu_8 ,\kappa_1 ,\kappa_2 , f ,g ) \setminus \{ 0 \} \mbox{ such that }  \nonumber \\
& \qquad (\nu '_1 , \nu '_2 , \nu '_3 , \nu '_4 ,\nu '_5 ,\nu '_6 , \nu '_7 ,\nu '_8 ; \kappa '_1 ,\kappa '_2 ; f' ,g' )  \nonumber \\
& \qquad \qquad =  (c \nu _1, c \nu _ 2, c\nu_3 , c \nu_4 , c \nu_5, c \nu_6 , c \nu_7, c \nu_8 ; c^2 \kappa_1 , c^2 \kappa_2 ; cf , c g )  .  \nonumber
\end{align}
This equivalence is related to equation (\ref{eq:rT1act}) on the transformation $T_1$.
We have the following statement on the representation of $\widetilde{W}(E^{(1)}_{8})$ in equation (\ref{eq:AffineWeylActionE8}) and the equivalence, which is proved similarly to Theorem \ref{thm:E6compati}.
\begin{thm}
Let $(\nu '_1 , \nu '_2 , \nu '_3 , \nu '_4 ,\nu '_5 ,\nu '_6 , \nu '_7 ,\nu '_8 ; \kappa '_1 ,\kappa '_2 ; f' ,g' ) $ be a tuple of the parameters which is equivalent to $ (\nu_1 , \nu_2 , \nu_3 , \nu_4 ,\nu_5 ,\nu_6 , \nu_7 ,\nu_8 ;\kappa_1 ,\kappa_2 ; f ,g )$.
The actions of the generators of  $\widetilde{W}(E^{(1)}_8)$ to $(\nu '_1 , \nu '_2 , \nu '_3 , \nu '_4 ,\nu '_5 ,\nu '_6 , \nu '_7 ,\nu '_8 ; \kappa '_1 ,\kappa '_2 ; f' ,g' ) $ are written as the same form of the actions to $(\nu _1 , \nu _2 , \nu _3 , \nu _4 ,\nu _5 ,\nu _6 , \nu _7 ,\nu _8 ; \kappa _1 ,\kappa _2 ; f ,g ) $ up to the equivalence.
\end{thm}

The difference between the time evolution of $q$-$P(E^{(1)}_8)$ and the particular parallel translation $T_1$ of $\widetilde{W}(E^{(1)}_{8})$ was described by $\Xi $ in equation (\ref{eq:E8Xi}), and we have
\begin{align}
& ( \Xi (\nu_1) , \Xi (\nu_2 ) , \Xi (\nu_3 ) , \Xi (\nu_4 ) , \Xi (\nu_5 ) , \Xi (\nu_6 ) , \Xi (\nu_7 ) , \Xi (\nu_8 ) ; \Xi (\kappa_1) , \Xi (\kappa_2 ) ; \Xi (f ) , \Xi (g ) )  \\ 
& \sim (\nu_1 , \nu_2 , \nu_3 , \nu_4 ,\nu_5 ,\nu_6 , \nu_7 ,\nu_8 ;\kappa_1 ,\kappa_2 ; f ,g ) \nonumber
\end{align}
by the equivalence in equation (\ref{eq:equivE8}).
Namely, we have
\begin{prop}
The transformation $\Xi $ in Theorem \ref{thm:E8TW} is within the equivalence of the parameters in equation (\ref{eq:equivE8}).
\end{prop}

\appendix
\section{A Lax pair of $q$-$P(E^{(1)}_6)$ and symmetry}
It is known that the $q$-Painlev\'e equation $q$-$P(E^{(1)}_6)$ is obtained as the compatibility condition for the Lax pair $L_1$ and $L_2$, where
\begin{align}
    L_{1}=&\frac{ z (g\nu_{1}-1)(g\nu_{2}-1)(g\nu_{3}-1)(g\nu_{4}-1)}{g(fg-1) ( g z - q )}-\frac{( g\kappa_{2}/\nu_{5}-1 )( g\kappa_{2}/\nu_{6}-1 ){\kappa_{1}}^{2}}{q fg \nu_{7}\nu_{8}} \label{eq:E6L1} \\
& +\frac{(\nu_{1}- z/q )(\nu_{2}- z/q )(\nu_{3}- z/q )(\nu_{4}- z/q )}{f- z/q}\Bigl\{\frac{g}{1-g z/q}-T_{z}^{-1}\Bigr\} \nonumber \\
    &+\frac{( \kappa_{1}/\nu_{7}-z )( \kappa_{1}/\nu_{8}-z )}{q(f-z)}\Bigl\{\Bigl(\frac{1}{g}-z\Bigr)-T_{z}\Bigr\} \, , \nonumber  \\
    L_{2}=&\Bigl(1-\frac{f}{z}\Bigr)T+T_{z}-\Bigl(\frac{1}{g}-z\Bigr) \, . \nonumber 
\end{align}
Here, $T_{z}$ represents the transformation $z\mapsto qz$ and $T$ represents the time evolution $T(f)=\overline{f}$.
In \cite{TsqP}, a symmetry of the linear $q$-difference equation $L_1 y(z)= 0$, which is written as
\begin{align}
& \Bigl\{ \frac{ z (g\nu_{1}-1)(g\nu_{2}-1)(g\nu_{3}-1)(g\nu_{4}-1)}{g(fg-1) ( g z - q )}-\frac{( g\kappa_{2}/\nu_{5}-1 )( g\kappa_{2}/\nu_{6}-1 ){\kappa_{1}}^{2}}{q fg \nu_{7}\nu_{8}} \Bigr\} y(z) \label{eq:E6L1yz} \\
& +\frac{(\nu_{1}- z/q )(\nu_{2}- z/q )(\nu_{3}- z/q )(\nu_{4}- z/q )}{f- z/q} \Bigl( \frac{g}{1-g z/q} y(z) -y(z/q) \Bigr) \nonumber \\
& +\frac{( \kappa_{1}/\nu_{7}-z )( \kappa_{1}/\nu_{8}-z )}{q(f-z)}\Bigl( \Bigl(\frac{1}{g}-z\Bigr) y(z)- y(qz) \Bigr) =0 , \nonumber 
\end{align}
was investigated.
Set $z =u/a$, $y(z)= z^d \tilde{y}(u) $.
Then, the function $\tilde{y}(u)$ satisfies
\begin{align}
 & \left\{  \frac{  u (g \nu _{1}-1)(g \nu _{2}-1)(g \nu _{3}-1)(g \nu _{4}-1)}{q (g/a) (fg-1)( u g /(a q ) -1 )}- a^2 \frac{ ( g \kappa_{2} /\nu_{5} -1 ) ( g \kappa_{2}/\nu_{6} - 1) {\kappa_{1}}^{2}}{q f g \nu_{7} \nu_{8}} \right\} \tilde{y}(u) \label{eq:E6tiy}\\
& + \frac{ ( a \nu_{1}- u/q )(a \nu_{2}- u/q )(a \nu_{3}- u/q )(a \nu_{4}- u/q ) }{a f- u/q}\left\{\frac{ g/a }{1-u g/(a q) } \tilde{y}(u) - \frac{\tilde{y}(u/q)}{a q^{d}} \right\} \nonumber \\
  & +\frac{( a \kappa_{1}/\nu_{7} -u )( a \kappa_{1}/\nu_{8} -u )}{q(a f-u)}\left\{\left(\frac{a}{g}-u \right) \tilde{y}(u) - a q^d \tilde{y}(q u) \right\} =0 . \nonumber 
\end{align}
If $a q^d =1$, then equation (\ref{eq:E6tiy}) is written in the form of equation (\ref{eq:E6L1yz}), where the parameters are changed as 
\begin{align}
& (\nu_1 , \nu_2 , \nu_3 , \nu_4 , \nu_5/\kappa_2 ,\nu_6/\kappa_2 , \nu_7/\kappa_1 ,\nu_8/\kappa_1,f,g) \label{eq:E6S} \\
& \mapsto (a  \nu_1 , a \nu_2 , a \nu_3 , a \nu_4 , \nu_5/\kappa_2/a  ,\nu_6/\kappa_2 /a , \nu_7/\kappa_1 /a ,\nu_8/\kappa_1 /a ,af,g/a) . \nonumber
\end{align}
We denote the transformation of the parameters in equation (\ref{eq:E6S}) by $S_{E_6}[a] $.
There are ambiguity for determining $(\nu _5, \nu _6, \kappa _2)$ (resp.~$(\nu _7, \nu _8, \kappa _1)$) from $(\nu_5/\kappa_2 ,\nu_6/\kappa_2  )$ (resp. $(\nu _7/\kappa _1 , \nu _8 /\kappa _1)$).
The definition of the equivalence in equation (\ref{eq:equivE6}) on the case $E^{(1)}_6$ is based on equation (\ref{eq:E6S}) with the ambiguity.
The equivalence in equation (\ref{eq:equivE7}) on $E^{(1)}_7$ (resp.~equation (\ref{eq:equivD5}) on $D^{(1)}_5$) is related to the similar symmetry in \cite[(4.7)]{TsqP} (resp.~\cite[(2.10), (2.12)]{TsqP}).

On the transformation $S_{E_6}[a] $ defined in equation (\ref{eq:E6S}) and the transformation $\Xi $ in equation (\ref{eq:E6Xi}), we have
\begin{equation}
 \Xi (\zeta ) = S_{E_6} \Bigl[  \frac{ \nu_ 5 \nu_ 6 }{\kappa_2 }  \Bigr](\zeta ). \label{eq:XiSE6}
\end{equation}
for $\zeta \in \{ \nu_1, \nu_2, \nu_3 ,  \nu_4, \nu_5 /\kappa _2, \nu_6 /\kappa _2,  \nu_7 /\kappa _1, \nu_8 /\kappa _1, f, g \} $.
There are a mistake in \cite[Theorem 3.1 (ii)]{TsqP},
and the equation (3.14) in \cite{TsqP} should be replaced to equation (\ref{eq:XiSE6}).

\section*{Acknowledgements}
The second author thanks Professor Anton Dzhamay for valuable comments.
He was supported by JSPS KAKENHI Grant Number JP22K03368.

%

\end{document}